\documentclass[12pt,letterpaper]{amsart}
\usepackage[utf8]{inputenc}

\usepackage{verbatim}
\usepackage{amsmath,amssymb,amsfonts,amsthm}
\usepackage{bigints}
\usepackage{enumitem}
\usepackage{hyperref}
\usepackage{array}

\usepackage{amsthm}

\usepackage{mathtools}

\mathtoolsset{showonlyrefs}

\newtheorem{definition}{Definition}[section]
\newtheorem{proposition}[definition]{Proposition}
\newtheorem{theorem}[definition]{Theorem}

\newtheorem{lemma}[definition]{Lemma}
\newtheorem{corollary}[definition]{Corollary}
\newtheorem{remark}[definition]{Remark}
\theoremstyle{plain}
\newtheorem*{Problem A}%[definition]
{Problem A}

\newtheorem*{Problem B}%[definition]
{Problem B}

\newtheorem*{Problem C}%[definition]
{Problem C}

{Problem}
\theoremstyle{definition}

\newcommand\R{\mathbb{R}}
\newcommand\N{\mathbb{N}}

\newcommand\wh{\widehat}
\newcommand{\norm}[1]{ \left\lVert#1\right\rVert}
\newcommand{\mathbbm}[1]{\text{\usefont{U}{bbm}{m}{n}#1}}

\author{}
\date{}
\begin{document}
\title{%1. The extremal  boundedness property of the Fourier operator
%\\Fourier basis is extremal (has the worst properties)
%An extremal property of the Fourier operator and applications 
%\\Fourier transform is an extremal bounded operator
%Is  Fourier transform the worst linear bounded operator?
%\\2.
%Extremality of the Fourier transform as a bounded operator
%\\
%Extremality of the Fourier transform   in a class of bounded operators
%\\Weighted Fourier inequalities by transference of boundedness
%\\The Fourier testing condition for linear operators
%\\
The Fourier transform is an extremizer of a class of bounded operators }

 %AVERAGED INTEGRAL TRANSFORMS

\thanks{M. Saucedo is supported by  the Spanish Ministry of Universities through the FPU contract FPU21/04230.
 S. Tikhonov is supported
by PID2023-150984NB-I00, 2021 SGR 00087. This work is supported by
the CERCA Programme of the Generalitat de Catalunya and the Severo Ochoa and Mar\'ia de Maeztu
Program for Centers and Units of Excellence in R\&D (CEX2020-001084-M)}

\author{Miquel Saucedo}
\address{M.  Saucedo,  Centre de Recerca Matemàtica\\
Campus de Bellaterra, Edifici C
08193 Bellaterra (Barcelona), Spain}
\email{miquelsaucedo98@gmail.com }

\author{Sergey Tikhonov}
\address{S. Tikhonov, Centre de Recerca Matem\`{a}tica\\
Campus de Bellaterra, Edifici C
08193 Bellaterra (Barcelona), Spain;
ICREA, Pg. Lluís Companys 23, 08010 Barcelona, Spain,
 and Universitat Autònoma de Barcelona.}
\email{ stikhonov@crm.cat}


\begin{abstract}
We show that, for a natural class of \textit{rearrangement admissible} spaces $X$ and $Y$, the Fourier operator is bounded between $X$ and $Y$ if and only if any operator of joint strong type $(1,\infty; 2,2)$ is also bounded between $X$ and $Y$. 

By using this result, we fully characterize 
the weighted Fourier inequalities of the form 
$$\qquad\qquad\norm{\widehat{f}u}_q   \leq C %\left(\int_{\R^d} |v(x){f}(x)|^p  dx \right)^{1/p} =:C 
\norm{fv}_p,\quad 1\leq p\leq \infty,\,0<q\le \infty,$$
for radially monotone weights $(u,v)$.
%This answers a long-standing problem posed by Benedetto--Heinig--Muckenhoupt forty years ago.
%Such a characterization  was known  only in the case of $p\le q$.
%and   it was an open question in the case  $q<p$.
This answers a long-standing problem posed by Benedetto--Heinig, Jurkat--Sampson, and Muckenhoupt. In the case of $p\le q$,
such a characterization  has been known since the 1980s.

\end{abstract}

\maketitle

%\tableofcontents

\section{Introduction}
\subsection{Motivation}
We define the Fourier series of a summable sequence $(f(n))_{n\in \mathbb{Z}}$ by
\begin{equation}\label{sumF}
\widehat{f}(x)=\sum_{n\in \mathbb{Z}} f(n) e^{2 \pi i nx}, \quad 0\leq x \leq 1.
\end{equation} 
We begin by recalling the basic
 yet crucial 
 results for Fourier series, namely the triangle inequality and Parseval's theorem:
\begin{align}
\label{ineq:l1}
    &\|\widehat{f}\,\|_{L^\infty} \leq \norm{f}_{\ell^1},\\
    \label{ineq:l2}
    &\|\widehat{f}\,\|_{L^2}\leq \norm{f}_{\ell^ 2}.
\end{align}
The Hausdorff-Young inequality extends 
 these results to  %the whole range of exponents 
  $p\in(2,\infty)$:
%  The extension of these results to  the whole range of exponents $p\in(2,\infty)$ is known as the Hausdorff-Young inequality
\begin{align}
\label{ineq:hy}
    &\|\widehat{f}\,\|_{L^p}\leq \norm{f}_{\ell^{p'}}.
\end{align}
It is worth mentioning that, although the standard proof derives \eqref{ineq:hy} from \eqref{ineq:l1} and \eqref{ineq:l2} through interpolation, the original arguments by Young and Hausdorff were different, relying on particular properties of the Fourier coefficients.
First, Young \cite{young} established \eqref{ineq:hy} for even values of $p$ by combining inequality \eqref{ineq:l2} with the formula $(\widehat{f}\,)^2=\widehat{f*f}$; Hausdorff \cite{hausdorff} subsequently extended Young's result to all $p$ % real numbers 
 by a careful analysis of the extremizers of \eqref{ineq:hy}. The modern interpolative proof, due to M. Riesz \cite{riesz} and
 based solely on \eqref{ineq:l1} and \eqref{ineq:l2},
  appeared 15 years after Young's first result. 
 
Interestingly, a similar course of events took place in the development of the Hardy-Littlewood-Pitt inequality
\cite{hardylittlewood,pitt} 
\begin{align}
\label{ineq:pitt2}
\big\||x|^{-\lambda} \widehat{f}\,\big\|
%\norm{\widehat{f} x^{-\lambda}}
_{L^q}\lesssim \norm{ f(1+|n|)^\alpha}_{\ell^p},
\end{align}
where $\lambda=\frac 1p +\frac 1q+ \alpha -1$, $1<p\leq q<\infty$, and $0\leq \alpha<\frac{1}{p'}$.
 Hardy-Littlewood and Pitt originally established  \eqref{ineq:pitt2} through complex  arguments relying on  properties of the Fourier coefficients. 
 Decades later, Zygmund and Stein \cite{stein} discovered that \eqref{ineq:pitt2} could be derived from \eqref{ineq:l1} and \eqref{ineq:l2} using interpolation.

 Finally, we note that  an interpolation argument \cite{hunt}  refines 
 %  \cite{hunt} that the interpolative method can be refined further to yield the following improvement of 
 \eqref{ineq:hy} and 
 \eqref{ineq:pitt2}
 in the Lorentz space scale:
\begin{align}
\label{ineq:hylor}
    &\|\widehat{f}\,\|_{L^{p,q}}\lesssim \norm{f}_{\ell^{p',q}},\quad 1<p<2, \;1\leq q\leq \infty.
\end{align}

Nowadays, the well-established  approach for deriving estimates of the forms
\begin{align}
\label{ineq:pittgen}
\big\|u\widehat{f}\,\big\|
_{q}&\lesssim \norm{ f v}_{p}, \quad \text{ with $u,v^{-1}$ even, non-increasing weights;}\\
\label{eq:pittgenX}
\big\|\widehat{f} \,\big\|
_{X}&\lesssim \norm{ f }_{Y}, \quad \text{with rearrangement-invariant function spaces $X,Y$}
\end{align} 
is Calderón's interpolative method 
\cite{Calderon}. 
This technique has been applied to the problem at hand by Jodeit--Torchinsky and 
Benedetto--Heinig--Johnson, see the survey \cite{BHJFAA} and Subsection \ref{section:preliminary} for more details. Importantly, every known inequality of  type \eqref{ineq:pittgen} or \eqref{eq:pittgenX}, in particular  \eqref{ineq:pitt2}
and 
\eqref{ineq:hylor}, can be derived through this method. This fact means that, up to now, every established Fourier estimate of the  type \eqref{ineq:pittgen} or \eqref{eq:pittgenX} is equivalent to the more general facts
\begin{align}
\label{ineq:Tpittgen}
\norm{u T(f)}_{q}&\lesssim \norm{ f v}_{p}, \quad \text{ with $u,v^{-1}$ even, non-increasing weights;}\\
\label{eq:TpittgenX}
\norm{T(f)}_{X}&\lesssim \norm{ f }_{Y}, \quad \text{with $X,Y$ rearrangement-invariant function spaces, }
\end{align}
respectively; where $T$ is any operator of joint strong type $(1,\infty;2,2)$, that is, it satisfies the inequalities
\begin{align}
\label{ineq:Tl1}
    &\norm{T({f})}_{\infty} \lesssim \norm{f}_{1},\\
    \label{ineq:Tl2}
    &\norm{T({f})}_{2}\lesssim \norm{f}_{2}.
\end{align}
\subsection{Statement of the problem}
\label{statement}
In light of all these facts, it is natural to wonder whether {\it all} Fourier inequalities of  type \eqref{ineq:pittgen} or \eqref{eq:pittgenX} always follow from the basic inequalities \eqref{ineq:l1} and \eqref{ineq:l2}.
In other words, does the boundedness of the Fourier operator imply the same for any other operator? 
More precisely, we state
\begin{Problem A}
   Do inequalities \eqref{ineq:pittgen} and \eqref{eq:pittgenX} always imply
    \eqref{ineq:Tpittgen} and \eqref{eq:TpittgenX}, respectively, for every operator of joint strong type $(1,\infty;2,2)$?
\end{Problem A}

Several remarks are  in order here.
\\
1. It is clear that if either the rearrangement invariance condition or the monotonicity assumption on $u$ and $v$ is removed, one can easily find instances of Fourier inequalities that cannot be obtained from \eqref{ineq:l1} and \eqref{ineq:l2}. For example, setting $$v(n)=\begin{cases}
    1,\quad & n\in \{2^k:k \in \mathbb{N}\}\\
    |n|+1,\quad &\text{otherwise},
\end{cases} $$ from the properties of lacunary Fourier series it is immediate to deduce the inequality
$$\norm{\widehat{f}\,}_{L^4}\lesssim \norm{vf}_
{\ell^{2}}.$$ However, there are operators of joint strong type $(1,\infty;2,2)$, such as the "rearranged" Fourier series operator $$T(f)(x)=\sum_{k\in \mathbb{N}} f({2^k}) e^{2 \pi i k x},$$  for which the inequality
$$\norm{T(f)}_{L^4}\lesssim \norm{vf}_
{\ell^{2}}$$ fails.
\\
2. There exists (see \cite{nielsen} and \cite[Theorem 3.3]{nielsen2}) an orthonormal, bounded basis $(\psi_k)_{k=0}^\infty$ of $L^2$ which satisfies a stronger Hausdorff-Young inequality
\begin{equation}
    \norm{\sum_{k=0}^\infty f(k) \psi_k}_p \lesssim_{p,\varepsilon} \norm{f}_{\ell^{2-\varepsilon}}, \mbox{ for any } \varepsilon>0 \mbox{ and }1<p<\infty.
\end{equation}
The latter  is clearly false for the Fourier basis.
%Second, it is in general not true that if a given operator of joint strong type $(1,\infty;2,2)$ is bounded between a pair of rearrangement-invariant spaces $X$ and $Y$, then every operator of joint strong type $(1,\infty;2,2)$ must also be bounded between $X$ and $Y$.
Thus, a positive answer to Problem A would imply that the Fourier basis is in a certain sense extremal among all orthonormal bounded bases. %??FIND SOME INTERESTING EXAMPLE BESIDES HARDY AND LAPLACE. IDEALLY ORTHONORMAL BOUNDED BASIS DEMOCRATIC BASIS An example of an almost greedy uniformly bounded
\\
3. 
Olevskii  \cite[page 59]{olevskiibook} noted that 
"in a number of problems [related to almost everywhere convergence of Fourier expansions] the Haar
system has the best properties among all complete orthogonal systems. Roughly
speaking, if a Fourier expansion divergence phenomenon occurs with the Haar system,
then such a phenomenon is unavoidable for any complete orthonormal system (or basis)."
Likewise, % We can informally think of 
Problem A 
can be understood  as the question of  whether the Fourier basis has the worst boundedness properties among all  orthonormal, bounded systems.
\\
4. %Analogous considerations
The previous discussion also 
applies to the Fourier coefficient operator and the Fourier transform, namely, to the operators (which we also denote by $\widehat{f}$)
\begin{equation}
    f \mapsto \left(\widehat{f}(n)\right)_{n\in \mathbb{Z}},
\end{equation}
where
$$\widehat{f}(n)=\int_0^1 f(x) e^{-2 \pi i x n} dx;$$ and 

\begin{equation*}\label{intF}
\wh f(\xi)=\int_{\mathbb{R}^d} {f}(x) e^{-2\pi i \langle x, \xi \rangle} d x.
\end{equation*} 
\begin{comment}
For both of these operators, we reformulate Problem A as follows:
\begin{Problem A}[Coefficient/Transform]
Do the inequalities
   \begin{align}
\label{ineq:pittgen2}
\norm{u \widehat f}_{q}&\lesssim \norm{ f v}_{p}, \quad \text{ with $u,v^{-1}$ even, non-increasing weights, and}\\
\label{eq:pittgenX2}
\norm{\widehat f}_{X}&\lesssim \norm{ f }_{Y}, \quad \text{with $X,Y$ rearrangement-invariant function spaces }
\end{align} imply
   \begin{align}
\label{ineq:Tpittgen2}
\norm{u T(f)}_{q}&\lesssim \norm{ f v}_{p}, \quad \text{ with $u,v^{-1}$ even, non-increasing weights, and also}\\
\label{eq:TpittgenX2}
\norm{T(f)}_{X}&\lesssim \norm{ f }_{Y}, \quad \text{with $X,Y$ rearrangement-invariant function spaces, }
\end{align}respectively, for every operator $T$ of joint strong type $(1,\infty;2,2)$?
\end{Problem A}

\end{comment}
\\
5.
In several particular cases the answer to Problem A is known to be positive. Let us give some examples.

First, for Fourier coefficients,
% To give a few examples, we first mention that in the case of coefficients, 
 inequality \eqref{eq:pittgenX} implies \eqref{eq:TpittgenX}
when  $Y=L^{p,q}([0,1])$ with $p<2$.
This follows from the Hardy-Littlewood theorem for monotone coefficients \cite{hunt}.
 Second, from the Kahane-Katznelson-de Leeuw theorem \cite{KKdL},
 the same holds for any space $Y$ such that
 ${L^\infty\subset Y\subset L^2}$. 
 We emphasize  that the result is not known for $L^{2,q}$ when $q>2$.

%or ${L^\infty\subset Y\subset L^2}$, the first case follows from the Hardy-Littlewood Theorem for monotone Fourier coefficients \cite{hunt}; and the second one, from the Kahane-Katznelson-de Leeuw Theorem \cite{KKdL}. We note that the result is not known for $L^{2,q}$ when $q>2$.
Third, for Fourier transforms, Benedetto and Heinig  \cite{BHJFAA} proved that when $1\leq p\leq q\leq \infty$
inequality \eqref{ineq:pittgen} implies \eqref{ineq:Tpittgen}. More precisely, they showed that %, for these parameters,
\begin{align}
    \eqref{ineq:pittgen} \implies \norm{u \int_0^{x^{-1}} f}_q\lesssim \norm{fv}_p \implies \eqref{ineq:Tpittgen},
    \qquad 1\leq p\leq q\leq \infty.
\end{align}
We would also like to point out the recent article \cite{kerman}, where it is shown that whenever $X$ is an interpolation space between $L^\infty$ and $L^2$, the chain of implications
\begin{align}
    \eqref{eq:pittgenX} \implies   \norm{ \int_0^{x^{-1}} f}_X\lesssim \norm{f}_Y \implies \eqref{eq:TpittgenX}
\end{align} holds true. Importantly, the second implication is known to fail in general, see for instance \cite{sin}.

Our main result, Theorem \ref{theorem:mainabstract}, gives an affirmative answer to Problem A in its full generality.
The solution to Problem A proceeds through the detailed study of \eqref{ineq:pittgen} in the extreme case $q\leq 1<p=\infty$, for which a Maurey-type factorization result  is employed.

\subsection{Notation}
We normalize the Lebesgue measure on $\mathbb{R}^d$ so that the measure of the unit ball is $1$. We let $\mathbbm{1}_E$ denote the characteristic function of the set $E$. 
Also,
$\mathbbm{1}_E^{-1}$ stands for 
$1/\mathbbm{1}_E.$
For a cube $A$  in $\R^d$ we denote its sidelength by $|A|$.

For a function $h: \mathbb{R}^d \to \mathbb{C}$ we denote its one-dimensional symmetric non-increasing rearrangement by $h^*$ and its multidimensional symmetric non-increasing rearrangement by $h^\circ$, i.e., the unique radial and non-increasing functions 
defined on $\mathbb{R}$ and $\mathbb{R}^d,$ respectively,
which are equimeasurable with $|h|$. 
We also set $f^{**}(t)=\frac1t\int_0^ t f^*$.  Observe that since a sequence $F:\mathbb{Z}\to \mathbb{C}$ can be regarded as a function $F(t):=\sum_{n\in \mathbb{Z}} F(n)\mathbbm{1}(t)_{[n-\frac 12, n+ \frac 12]}$, all these notions extend naturally to sequences.
%and  $a_n^{**}(t)=\frac1n\sum_{j=1}^ n a_j^* $, respectively

For a radial function $h$,
we denote its radial part  by $h_0$, that is, $h(x)=h_0(|x|)$ for $x\in\mathbb{R}^d$. 
We observe that 
$h^*(x)=(h^\circ)_0(|x|^{\frac{1}{d}})$ for $x\in \mathbb{R}^d$.

We 
say that a function $h$  is radial monotone if it is radial and $h_0$ is monotone.
Thus, if $h$ is radial and non-increasing, $h^*(x)=h_0(|x|^{\frac{1}{d}})$ for $x\in \mathbb{R}$  and $h^\circ(x)=h(x)$ for $x\in \mathbb{R}^d$. We also define the non-decreasing symmetric rearrangement of $h$, $h_*$,  by $1/h_* = (1/h)^*$.

We use the following special notation for powers of rearrangements:
\begin{equation*}
    f^{*,p}(x)
:=(f^{*}(x))^p,\quad f^{\circ,p}(x)
:=(f^{\circ}(x))^p,\quad f_{*}^p(x)
:=(f_{*}(x))^p.
\end{equation*}
For a weight function $v$ satisfying $0\leq v\leq \infty$,
%$\norm{fv}_p:=\Big(\int |fu|^p\Big)^\frac1p$ %is defined in \eqref{ineq:pittgen} 
 %and
 we let $L^p(v):=\Big\{f: \norm{fv}_p=\Big(\int |fv|^p\Big)^\frac1p<\infty\Big\}$.
As usual, we understand $\Big(\int_I |fv|^p\Big)^\frac1p$ for $p=\infty$
as $\sup_I |fv|.$

For  $0<p,q\le\infty$, we define $p', r$, and $p^{\sharp}$ through the relations 
$$\frac1p+\frac1{p'}=1,\quad\frac1r:= \frac1q-\frac1{p},\quad\mbox{ and}\quad \frac{1}{p^{\sharp}}:= \left|\frac12-\frac1{p}\right|.
$$

We  use the symbol $F \lesssim G$  to mean that $F \le K G$ with a constant $K=K(d)$ that may change from
line to line.
If the constant $K$ depends on a given parameter $\lambda$ other than the dimension, 
then we write
$F\lesssim _{\lambda} G$.
The symbol $F \approx G$ (similarly, $F \approx_\lambda G$) means that both $F \lesssim G$ and $G \lesssim F$ hold. 
\subsection{Main results}
Before stating our main results, we give some definitions.
\begin{definition}
    Let $\norm{\cdot}_X$ be a non-negative functional defined on functions or sequences.
    We say that $f\in X$ if $\norm{f}_{X}<\infty$.
    We define its associate functional $\norm{\cdot}_{X'}$ by
 \begin{equation*}
        \norm{f}_{X'}:=\sup \left\{\int _{\mathbb{R}^d}\left | fg \right| : \norm{g}_X\leq 1\right\}
    \end{equation*} in the function case, and
 \begin{equation*}
        \norm{f}_{X'}:=\sup \left\{\sum_{n\in \mathbb{Z}} \left | f(n)g(n) \right| : \norm{g}_X\leq 1\right\}
    \end{equation*} in the sequence case.
\end{definition}

\begin{definition} We say that $\norm{\cdot}_X$ is left-admissible if the following properties hold:
\begin{enumerate}[label=(\roman*)]\item  for any $F \in X''$, $\norm{F}_X= \norm{F}_{X''};$

    \item  for any $h\in X'$, $\norm{h^\circ}_{X'}\leq \norm{h}_{X'}.$
\end{enumerate}
   
\end{definition}

\begin{definition} We say that $\norm{\cdot}_Y$ is right-admissible if the following properties hold:
\begin{enumerate}[label=(\roman*)]

\item  for any $G\in Y$ and $m\in L ^\infty$, $\norm{m G}_Y \leq \norm{G}_Y \norm{m}_\infty;$

    \item for any $G\in Y$, $\norm{G^\circ}_Y \leq \norm{G}_Y.$

\end{enumerate}
    
\end{definition}

\begin{remark}
\begin{enumerate}[label=(\roman*)]
%\item For $q\geq 1 $, if we set $\norm{F}_X=\norm{uF}_q$, then  $\norm{h}_{X'}=\norm{u^{-1}h}_{q'}$.
    \item Examples of left-admissible $\norm{\cdot}_X$ include 
$\norm{F}_X=\norm{Fu}_q$ for $1\leq q\leq  \infty$ and $u$ radial, non-increasing { (in this case $\norm{h}_{X'}=\norm{u^{-1}h}_{q'}$) }and also $\norm{F}_X=\rho(F)$ for  a rearrangement-invariant function space norm $\rho$.
\item Examples of right-admissible $\norm{\cdot}_Y$ include 
$\norm{G}_Y=\norm{Gv}_p$ for $0< q \leq \infty$ and $v$ radial, non-decreasing; $\norm{G}_Y=\norm{G^\circ v}_p$ for  $0< q \leq \infty$ and any $v$; and also $\norm{G}_Y=\rho(G)$  for  a rearrangement-invariant function space norm  $\rho$.
\item Under the additional condition that
$\norm{\cdot}_Y= \norm{\cdot}_{Y''}$, the associates of left- and right-admissible spaces are right- and left-admissible, respectively.
\item In the definitions of left- and right-admissible,
the symbols
$\leq$ and $=$ can be replaced by $\lesssim$ and $\approx$ at the expense of 
 potentially increasing 
 the constants in Theorem \ref{theorem:mainabstract}.
\end{enumerate}

\end{remark}

We are now in  a position to state  our main result.
\begin{theorem} 
\label{theorem:mainabstract}Let $\norm{\cdot}_X $ and $\norm{\cdot}_Y$ be left- and right-admissible, respectively, and $0< \beta\leq 1$.

   Then, the following are equivalent$:$
   \begin{enumerate}[label=(\roman*)] 
        \item the inequality \begin{equation}
   \label{ineq:pittXY}
       \norm{|\widehat{f}|^\beta}_X^{\frac1\beta} \leq C_1 \norm{ f}_Y
          \end{equation}
          holds for any $f\in L^1;$
 
   \item the inequality
   
          \begin{equation}
              \label{ineq:Tfxy}
              \norm{|T(f)|^\beta}_X^{\frac1\beta}\leq C_2 \norm{f}_Y
          \end{equation} holds for any $f\in L^1+L^2$ and  sublinear operator $T$ of joint strong type $(1,\infty;2,2)$.
          \end{enumerate}
          Moreover, the optimal constants satisfy $C_1 \approx_{\beta} C_2$.
\end{theorem}
\begin{remark}
    The parameter $\beta$ allows us to apply our result, for example, to inequality \eqref{ineq:pittgen} with $q<1$. Indeed, even though $f \mapsto \norm{f u}_q$ is not left-admissible for $q<1$, we have  $\norm{fu}_q=\norm{f^qu^q}_1^{\frac1q}$ and $f \mapsto \norm{f u^q}_1$ is left-admissible. 
\end{remark}
Theorem \ref{theorem:mainabstract} has the following immediate consequence:
\begin{corollary} 
    \label{coro:mainbis}
    For a left-admissible space $X$, define
    the rearrangement-invariant right-admissible functional

    \begin{equation*}
        \norm{f}_{\widetilde{X}}:= \sup_T %\norm{T(f)}_{X}
         \norm{|T(f)|^\beta}_X^{\frac1\beta},
    \end{equation*} where the supremum is taken over all sublinear operators of joint strong type $(1, \infty;2,2)$. Then,  
for any right-admissible $Y$ for which the inequality
    \begin{equation}
        \norm{|\widehat{f}|^\beta\,}^{\frac1\beta}_X \leq \norm{f}_Y
    \end{equation} holds for any $f\in L^1$, we have   $$\norm{f}_{\widetilde{X}}\lesssim_\beta\norm{f}_Y$$
     for any $f\in L^1$.
\end{corollary}
\subsection{Applications}
As a consequence of Theorem \ref{theorem:mainabstract}  we are able to answer the following long-standing question:
\begin{Problem B}
   Characterize the radial non-increasing weights $u$ and $v^{-1}$ for which the inequality 
   $\big\|u\widehat{f}\,\big\|
_{q}\lesssim \norm{ f v}_{p}$
   %   \eqref{ineq:pittgen}
    holds.
\end{Problem B}
Problem B was posed by Benedetto-Heinig, 
Jurkat–Sampson,
and Muckenhoupt in the 1980s, see
\cite{h,JS1,Muc1}. A complete solution of this problem is only known in the case  $p\le q$.

In the case $q<p$ and $\max(q,p')\geq 2$, sufficient conditions have been obtained in \cite{BHJFAA} and \cite{nachr}, however
%, no necessary conditions were known other than the trivial ones. 
%other than those from the case $p<q.$ However
 it is not known whether they are necessary.
Regarding the most difficult case $q<2<p$, not much is known, besides the fact that the conditions in \cite{BHJFAA} are no longer sufficient
unless the weights 
satisfy some additional conditions,
see \cite{sin}. 

We note that weighted Fourier inequalities have been used, among other applications (see, e.g., \cite{beckner2008, bui}), to obtain 
Carleman-type estimates  that are crucial  for studying unique continuation problems (see \cite{gradient}).

%unique continuation Carleman-type results for solutions of first order differential equations and systems, see \cite{gradient}.

%Since the characterization of such weights is quite lengthy, we give it in Section \ref{section:finalremarks}.

The solution to Problem B
is given in Section \ref{section:finalremarks} and 
proceeds through the explicit computation of $\widetilde{X}$ for  $\norm{|\widehat{f}|^\beta\,}_X^{\frac1\beta}= \norm{u\widehat{f}\,}_q,$ a result which is of independent interest
(see Corollary \ref{coro:optiHL}). 

Furthermore, by computing $\widetilde{X}$,  
 %As a result of computation of $\widetilde{X}$,  
% As an example of an application of the explicit computation of $\widetilde{X}$, 
  we find 
  (see Theorem \ref{th:52})
  that the smallest rearrangement invariant function space which contains the Fourier coefficients of every $f\in L^{2,p}(\mathbb{T})$ for $2<p\leq \infty$ is given by the norm $$
\left(\sum_{n=1}^\infty \left(\sum_{j=1}^n c_n^{*,2 } \right)^{\frac{p}{2}} \frac{1}{n \log^{\frac{p}{2}}(n+1)}\right)^{\frac{1}{p}},
$$ thus solving the problem of describing  optimal Fourier inequalities in Lorentz spaces mentioned in Subsection \ref{statement}.

 \subsection{Structure of the article}
%We have organized this article in the following way.

We now proceed to outline the structure of this work.
In Section \ref{section:aux} we collect some preliminary  results which will be used throughout the paper: Calderón's rearrangement method, duality principle, 
Maurey's factorization, Khinchin's inequality, and 
Hardy's inequalities.

In Section \ref{section:necessary} we obtain new necessary conditions on the general  (not necessarily monotonic) weights $u$ and $v$ for the inequality $\big\|u\widehat{f}\,\big\|
_{q}\leq C \norm{ f v}_p$
 %\begin{equation}
%\label{ineq:pittnomono}   \big\|u\widehat{f}\,\big\|_{q}\leq C \norm{ f v}_p
%\end{equation}
to hold for $f\in L^1$.
%We note that for the proof of Theorem \ref{theorem:mainabstract} it suffices to study inequality \eqref{ineq:pittgen} for $p=\infty$ and $q\leq 1$ with $u$ and $v^{-1}$ monotone, whenever possible we obtain necessary conditions for a wider range of $p,q$ under no monotonicity restrictions on the weights, as the more general proofs are not substantially longer.
We focus on the case of the Fourier transform; the results for series and coefficients can be obtained similarly.  

Section \ref{section:mainproofs} contains the proof of our main result, Theorem \ref{theorem:mainabstract}. In more detail, first we prove Theorem \ref{theorem:mainabstract} for $X=L^1(u)$ and $Y=L^\infty(v)$ by making use of results in Sections \ref{section:aux}  and \ref{section:necessary}. 
Then we obtain the result in full generality by combining the previous particular case with the properties of left- and right-admissible spaces.

In Section \ref{section:finalremarks} we  present the solution to
Problem B given in Theorem \ref{theorem:mainh}. In order to prove it, 
as an application of the main Theorem \ref{theorem:mainabstract},
we first 
find  the largest right-admissible space $Y$  for which
$\norm{u\widehat{f}\,}_q\leq \norm{f}_Y$
        holds with 
 radial non-increasing 
 $u$ (see 
 Corollary \ref{coro:optiHL}).
 %by combining  the main Theorem \ref{theorem:mainabstract} with the  Hardy's inequalities in %relying heavily on 
 %the results in Subsection \ref{subsec:hardy}. %Theorem \ref{theorem:mainh} follows from  Corollary \ref{coro:optiHL} and Hardy's inequalities.
We also discuss the corresponding results for Fourier series on $\mathbb{T}$.

 Finally, as a further application of Theorem \ref{theorem:mainabstract}, in Section  \ref{section:optimalineq}, 
 we 
prove
 optimal Fourier inequalities for 
 Lorentz, Morrey and  Orlicz spaces. 

\section{Auxiliary  results}
\label{section:aux}
In order to make this article  self-contained, 
in this section we collect several known results which will be needed in our proofs.

\subsection{The method of Calderón}
\label{section:preliminary} 
We now explain the Calderón-type interpolative technique to obtain inequalities \eqref{ineq:pittgen} and 
\eqref{eq:pittgenX}. Before giving the results, we introduce some notation.

\begin{definition}
\label{def:prec}
    We write $F\prec G$ whenever
    \begin{equation*}\int_0 ^x F^{*,2} dt  \leq\int_0 ^x \left(\int_0 ^{t^{-1}} G^* ds \right)^2 dt.
\end{equation*}
\end{definition}
The main idea of the method is the following characterization of sublinear operators of joint strong type $(1,\infty;2,2).$
\begin{lemma}[\protect {\cite{bennett1988interpolation} and \cite{jodeit}}]
\label{lemma:kfunc}
Let $T$ be a sublinear operator. Then $T$ is of joint strong type $(1, \infty;2,2)$ if and only if for any $x \in \mathbb{R}_+$ and $f\in L^1 +L^2$,
 \begin{equation}
 \label{ineq:Kfunc}
        \int_0 ^x T(f)^{*,2} \lesssim  \int_0 ^x \left(\int_0 ^{t^{-1}} f^*\right)^2 dt, %\approx x \left(\int_0 ^{x^{-1}} f^*\right)^2 + \int_{x^{-1}}^\infty f^{*,2},
    \end{equation}
   % Moreover, if $T(f)=\widehat{f}$, the same holds for $f\in W(L^1, \ell^2)$. 
% In particular, using the notation introduced in Definition \ref{def:prec}, the Fourier transform satisfies
% $$\widehat{f}\prec K_d f,$$ where $K_d$ is a dimensional constant.
which, in the notation of Definition \ref{def:prec}, means that the operator $T$ satisfies $$T(f)\prec K f$$ with some constant $K$.

Conversely, for any $F\prec G$ with $G\in L^1+L ^2$, there exists a sublinear operator of joint strong type $(1,\infty;2,2)$ such that $T(G)=F$.
\end{lemma}
\begin{proof}The first part of the result is given in \cite[Theorem 4.7]{jodeit}.
    For the proof of the converse we first find $T$ such that $T(G^*)=F^*$. Let $f=F^{*,2}$ and $g=\left(\int_0^{t^{-1}} G^{*}\right)^2$.
    By \cite[Chapter 3, Theorem 2.10]{bennett1988interpolation}, there exists a positive linear operator $S$ of joint strong type $(1,1;\infty,\infty)$, equivalently, with the property that for any $h:\mathbb{R} \to \mathbb{C}$
    \begin{align}
    \label{11}
      \int_0^x S(h)^{*}\lesssim  \int_0^x h^{*},
    \end{align}
    such that \begin{align*}
       S(g)=f.
    \end{align*}
    Define
    \begin{align*}
      \widetilde{T}(h)=\Big(S\Big(\Big[\int_0^{t^{-1}} h\Big]^2\Big)\Big)^\frac12
    \end{align*} and observe that $\widetilde{T}(G^*)=F^{*}$. Moreover,
  since $S$ is positive, 
  $\widetilde{T}$ is sublinear. Finally, from the property \eqref{11} we see that $\widetilde{T}$ satisfies condition \eqref{ineq:Kfunc}, whence we deduce that it is of joint strong type $(1,\infty;2,2)$.

  Next, once again by \cite[Chapter 3, Theorem 2.10]{bennett1988interpolation} there exist  positive linear operators $S_1$ and $S_2$ of joint strong type $(1,1;\infty,\infty)$ such that $S_1(F^*)=F$ and $S_2(G)=G^*$. Thus, $T=S_1 \circ \widetilde{T} \circ S_2$ is the desired operator.

  We note that if $F$ is a sequence,
  it can be interpreted as
$\sum_{n \in \mathbb{Z}} F(n) \mathbbm{1}_{[n-\frac 12, n+\frac 12]}$.
To ensure that the operator is sequence-valued, we define
 $$T_{seq}(h)=\sum_{n \in \mathbb{Z}} \left(\int_{n-\frac 12}^{n+\frac 12} T(h)\right) \mathbbm{1}_{[n-\frac 12, n+\frac 12]}.$$
\end{proof}
In light of Lemma \ref{lemma:kfunc}, inequality \eqref{ineq:pittXY} can be obtained by establishing an inequality for non-increasing functions/sequences. More precisely, 
\begin{corollary}
\label{coro:method} Let $X$ and $Y$ be left- and right-admissible spaces
and $0< \beta\leq 1$.
 Then, inequality \eqref{ineq:Tfxy} holds for any sublinear $T$ of joint strong type $(1,\infty;2,2)$ and any $f$ if and only if for any $F\prec G$
\begin{align}
\label{ineq:FGxy}
    \norm{F^{\circ,\beta}}_X^\frac1\beta\lesssim \norm{G^\circ}_Y.
\end{align} In particular, for monotone $u$ and $v^{-1}$, inequality \eqref{ineq:Tpittgen} holds for any $T$ of joint strong type $(1,\infty;2,2)$ and any $f$ if and only if for any $F\prec G$
\begin{align}
\label{ineq:FGuv}
    \norm{F^\circ u}_q\lesssim \norm{G^\circ v}_p.
\end{align}
\end{corollary}

\begin{proof}
    First, by Lemma \ref{lemma:kfunc}, there exists a constant $K$ such that setting $F=|T(f)|$ and $G=K|f|$, we have $F\prec G$. Thus, assuming that  inequality \eqref{ineq:FGxy} holds for any $F\prec G$,  by right-admissibility of $Y$ we have
    $$\norm{T(f)^{\circ , \beta}}_X^{\frac{1}{\beta}} \lesssim \norm{K f^\circ}_Y \lesssim K \norm{f^\circ}_Y.$$  Second, by left-admissibility of $X$, for any $F$ we have $\norm{F}_X\leq \norm{F^\circ}_X$. Indeed,
    \begin{eqnarray*}
        \norm{F}_X &=& \norm{F}_{X''} = \sup_{\norm{h}_{X'}\leq 1} \int  \left | F h \right| \leq \sup_{\norm{h}_{X'}\leq 1}  \int F^\circ h^\circ  \\ &\leq&\sup_{\norm{h}_{X'}\leq 1} \norm{F^\circ}_X \norm{h^\circ}_{X'}\leq \sup_{\norm{h}_{X'}\leq 1}\norm{F^\circ}_X \norm{h}_{X'} \leq \norm{F^\circ}_X,
    \end{eqnarray*}
    where the first inequality is the Hardy-Littlewood rearrangement inequality.
    Therefore, if inequality \eqref{ineq:FGxy} holds,  by right-admissibility of $Y$, for any $f \in L^1+L^2$, we have
    $$\norm{|T(f)|^{ \beta}}_X^{\frac{1}{\beta}} \leq \norm{T(f)^{\circ , \beta}}_X^{\frac{1}{\beta}}\lesssim
    \norm{f^\circ}_Y \lesssim \norm{f}_Y,$$ that is, inequality \eqref{ineq:Tfxy} is valid.

    The proof of the converse follows from the second part of Lemma \ref{lemma:kfunc}.
\end{proof}
\subsection{Duality}

We recall here that because the Fourier operator is self-adjoint, its boundedness between two spaces implies its boundedness between their respective associate spaces. 
%for $(u,v,p,q)$ and $(v^{-1},u^{-1},q',p')$. More precisely,
\begin{lemma}
\label{lemma:duality} Assume that $\norm{\cdot}_Y$ satisfies property (i) of right-admissibility. If
\begin{equation*}
    \big\|\widehat{f}\,\big\|_X\leq \norm{f}_Y
\end{equation*} holds for any $f\in L^1$, then, for any $g\in L^1$,
\begin{equation*}
    \norm{\widehat{g}}_{Y'}\leq \norm{g}_{X'}.
\end{equation*}In particular, 
for $1\leq p,q \leq \infty$,
 the inequality $\big\|u\widehat{f}\,\big\|
_{q}\leq C \norm{ f v}_p$ holds for any $f\in L ^1$, if and only if for any $g \in L^1$,
$%    \begin{equation*}\label{eq:dual}
    \norm{v^{-1} \widehat{g} }_{p'}\leq C \norm{u^{-1}g}_{q'}.
$%\end{equation*}
\end{lemma}
\begin{comment}
\begin{proof}[Proof to remove] If $f,g\in L^1$,
    $$|\int \hat{g} \bar{f}|= |\int {g} \bar{\hat{f}}| \norm{g}_{X'}\norm{\hat{f}}_X \leq \norm{g}_{X'}\norm{f}_Y.$$
    Thus, by (i), if $g\in L^1$,
    $$\norm{g}_{X'}\geq \sup_{f \in L^1, \norm{f}_Y=1}|\int \hat{g} \bar{f}|= \sup_{f \in L^1, \norm{f}_Y=1}\int |\hat{g} f|.$$ 
    Finally, once again by (i),  $$\norm{\hat{g}}_{Y'}= \sup_{\norm{f}_Y=1}\int |\hat{g} f| =\sup_{f c. support, \norm{f}_Y=1}\int |\hat{g} f|.$$ 
    
\end{proof}
\end{comment}

\subsection{Khinchin's inequality and factorization} 
%This  subsection contains some needed results from functional analysis. %which are mainly taken from \cite{albiac2006topics} and \cite{Lindenstrauss1996}.
%The Khinchin inequality is a simple tool which allows the use of randomization techniques in a variety of contexts.
\begin{lemma} [Khinchin's inequality,   \protect{\cite[Theorem 6.2.3]{albiac2006topics}}]
\label{theorem:khintchine}
For any $0<p<\infty$ we have, for any $n$ and $a_1, \dots, a_n \in \mathbb{C}$,
\begin{equation}\label{khi}
    %A_p \left(\sum_{j=1}^n |a_j|^2 \right)^{\frac{1}{2}} \leq 
    \left(\mathbb{E} \left [\left|\sum_{j=1}^n a_j \varepsilon_j\right|^p\right ] \right)^{\frac{1}{p}} \approx_p
    \left(\sum_{j=1}^n |a_j|^2 \right)^{\frac{1}{2}},
\end{equation} where 

\begin{equation*}\mathbb{E} \left [\left|\sum_{j=1}^n a_j \varepsilon_j\right|^p\right ] = 2^{-n} \sum_{\varepsilon_1= \pm 1, \dots, \varepsilon_n=\pm 1} \left|\sum_{j=1}^n a_j \varepsilon_j\right|^p.\end{equation*}
\end{lemma}

\begin{theorem}
    
\label{coro:usedfactor}
   Let $T$ be a linear operator, $u$ and $v$ be non-negative functions, and 
    $0<q\leq 2\leq p\leq \infty$. Assume that the inequality
    \begin{equation}
        \norm{u T(f)}_{q} \leq C \norm{fv}_p
    \end{equation} holds for any function $f\in L^p(v)$. Then, there exists $h$ with $\norm{h^{-1}}_{q^{\sharp}}=1$ such that, for any $f\in L^p(v)$,
    \begin{equation}
        \norm{u h T(f)}_2 \lesssim_{q} C \norm{fv}_p.
    \end{equation} 
\end{theorem} 
Even though this result is somewhat folklore, we have not been able to find an exact reference. For the readers' convenience,
we give a proof here.
The argument is based on the following factorization criterion due to Maurey:
    \begin{theorem}[\protect{\cite[Theorem 1]{Maurey1973}}] 
\label{theorem:factorization}
Let $E$ be a quasi-normed space and $0<q\leq 2$.
Assume that 
 $\widetilde{T}:E \to L^q$ is a bounded linear operator.
Then, the following are equivalent:
\begin{enumerate}[label=(\roman*)]
\item there exists $h$ with $\norm{h^{-1}}_{q^{\sharp}}=1$ such that for every $f\in E$ the inequality \begin{equation}
\label{eq:facto2}\norm{h\widetilde{T}(f)}_{2}\leq\norm{f}_E\end{equation} holds;
\item for every finite sequence $(f_n)_{n=1}^N$ from $E$ the inequality
\begin{equation}\label{vsp3}
    \norm{\left(\sum_{n=1}^N |\widetilde{T}(f_n)|^2\right)^{\frac{1}{2}}}_{q}\leq  \left(\sum_{n=1}^N\norm{ f_n}_E^2\right)^{\frac{1}{2}}
\end{equation} holds.
\end{enumerate}
\end{theorem}
\begin{proof}[Proof of Theorem \ref{coro:usedfactor}]
We show that for $E=L^p(v)$ with $p\geq 2$, item (ii) in Theorem \ref{theorem:factorization} is true for any bounded linear operator $\widetilde{T}:E \to L^q$. The desired $h$ is then obtained through item (i) in Theorem \ref{theorem:factorization} with $\widetilde{T}=uT$.

%For $2\leq p<\infty$ item (ii) follows easily from Khinchin's and Minkowski's inequalities (see \cite[Theorem 7.1.3]{albiac2006topics}).  

For $2 \leq p\leq \infty$ and $q\geq1$ we first observe  that $\big\|{\left(\sum_{n=1}^N |f_n|^2\right)^{\frac{1}{2}}}\big\|_{L^p(v)} \leq  \left(\sum_{n=1}^N\norm{ f_n}_{L^p(v)}^2\right)^{\frac{1}{2}}$.
Then \eqref{vsp3} follows from the 
 Grothendieck inequality for  arbitrary Banach lattices
(see \cite[1.a.1]{Lindenstrauss1996} for the definition of a Banach lattice):

 \begin{theorem}[General Grothendieck inequality, \protect{\cite[1.f.14] {Lindenstrauss1996}}]
\label{theorem:groth}
    Let $X$ and $Y$ be Banach lattices. Assume that $\widetilde{T}:Y \to X$ is a bounded linear operator. Then, for every finite sequence $(y_n)_{n=1}^N$ of elements from $Y$, 
    \begin{equation}
\label{eq:groth}\norm{\left(\sum_{n=1}^N |\widetilde{T}(y_n)|^2\right)^{\frac{1}{2}}}_X\leq K_G \norm{\widetilde{T}}_{Y \to X} \norm{\left(\sum_{n=1}^N |y_n|^2\right)^{\frac{1}{2}}}_Y,\end{equation} where $K_G$ is an absolute constant known as the universal Grothendieck constant.
\end{theorem}

 In the non-Banach case, that is, when $q< 1$,
 %For the case $p=\infty$ and $q< 1$,
  we now show that inequality \eqref{eq:groth}   holds true for $X=L^q(u)$ and $Y=L^p(v)$. By Step 2 and the first part of Step 3 in the proof of {\cite[1.f.14] {Lindenstrauss1996}}, it suffices to obtain inequality \eqref{eq:groth} for the finite sequence spaces $Y=\ell^\infty(\{1, \dots, N\})$ and $X= \ell^q(\{1, \dots, N\}).$
To do so, observe  
    \cite[p.~52]{lacey1980notes}
    that
     for any  finite matrix $A$ there exist finite matrices $B$ and $D$ such that $$A=DB$$ and \begin{equation}
    \label{eq:abc}
    \norm{A}_{\ell^\infty \to \ell^q} \approx _q \norm{D}_{\ell^1\to \ell^q} \norm{B}_{\ell^\infty \to \ell^1}. \end{equation}

    Hence, using Khinchin's and Jensen's inequalities we deduce that, for any finite sequence $(y_n)_{n=1}^N\subset Y$,
    \begin{align*}
       \Big\|{\left(\sum_{n=1}^N |A(y_n)|^2\right)^{\frac{1}{2}}}
       \Big\|^q _{\ell^q}&=\norm{\left(\sum_{n=1}^N |DB(y_n)|^2\right)^{\frac{1}{2}}}^q _{\ell^q}\approx_q   \mathbb E \norm{\sum_{n=1}^N \varepsilon _n DB(y_n)}^q _{\ell^q}  \\
        &%\quad\quad
        \leq  \mathbb E \norm{\sum_{n=1}^N \varepsilon _n B(y_n)}^q _{\ell^1} \norm{D}_{\ell^1\to \ell^q}^q\\
        & \leq \left( \mathbb E \norm{\sum_{n=1}^N \varepsilon _n B(y_n)}_{\ell^1}\right)^q  \norm{D}_{\ell^1\to \ell^q}^q\\
        &%\quad\quad
        \approx_q  \norm{\left(\sum_{n=1}^N  |B(y_n)|^2\right)^\frac{1}{2} }^q _{\ell^1} \norm{D}_{\ell^1\to \ell^q}^q \\
        &%\quad\quad
        \leq  K_G^q \norm{B}_{\ell^\infty\to \ell^1}^q \norm{D}_{\ell^1\to \ell^q}^q \norm{\left(\sum_{n=1}^N  |y_n|^2\right)^\frac{1}{2} }^q _{\ell^\infty},
    \end{align*} where in the last inequality we used Theorem \ref{theorem:groth} for $X=\ell^1$. The result now follows by using  \eqref{eq:abc}.
    \end{proof}

\subsection{Hardy's and related inequalities}
\label{subsec:hardy}
In light of Theorem \ref{theorem:mainabstract} and Corollary \ref{coro:method}, the solution of Problem B follows from the characterization of inequality \eqref{ineq:FGuv}. Thus, we will need  the characterizations of several Hardy type inequalities.
\begin{lemma}[\protect{see \cite[Theorem~7]{kufner2007hardy} for $\mathfrak{p}> 1$ and \cite[Theorem~8]{popova} for $\mathfrak{p}\leq 1$}]

\label{theorem:discretehardy}

Let $0<\mathfrak{p},\mathfrak{q}\leq \infty$ and
$\frac{1}{\mathfrak{r}}:= \frac{1}{\mathfrak{q}}-\frac{1}{\mathfrak{p}}$. Let $u,v\geq 0$. The smallest constant $K$ for which the inequality 
\begin{equation}
\label{eq:discretehar}
\left(\sum_{n=1}^\infty u_n \left(\sum_{j=n}^{\infty} x_j \right)^\mathfrak{q} \right)^{\frac{1}{\mathfrak{q}}} \leq K \left(\sum_{n=1}^\infty v_n x_n^\mathfrak{p} \right)^{\frac{1}{\mathfrak{p}}},\qquad \mathfrak{q}<\mathfrak{p},
\end{equation}
holds for any non-negative sequence $(x_n)_{n=1}^\infty $ satisfies
\begin{enumerate}[label=(\roman*)]
   
\item for $\mathfrak{p}\leq 1$,
\begin{equation}
\label{eq:discretehardyp-}
    K\approx_{\mathfrak{p},\mathfrak{q}} \left(\sum_{n=1}^\infty u_n \left(\sum_{j=1}^n u_j \right)^{\frac{\mathfrak{r}}{\mathfrak{p}}} \left(\inf_{j\geq n} v_j \right)^{-\frac{\mathfrak{r}}{\mathfrak{p}}}\right)^{\frac{1}{\mathfrak{r}}};
\end{equation}
\item for $1<\mathfrak{p}$, 
\begin{equation}
\label{eq:discretehardyp+}K\approx_{\mathfrak{p},\mathfrak{q}} \left(\sum_{n=1}^\infty u_n \left(\sum_{j=1}^n u_j \right)^{\frac{\mathfrak{r}}{\mathfrak{p}}} \left(\sum_{j=n}^\infty v_j^{\frac{1}{1-\mathfrak{p}}} \right)^{\frac{\mathfrak{r}}{\mathfrak{p}'}}\right)^{\frac{1}{\mathfrak{r}}}.\end{equation}
\end{enumerate}
%where $r^{-1}=q^{-1}-p^{-1}$.
\end{lemma}
\begin{lemma}[\protect{\cite[Theorem~5]{kufner2007hardy}}]

\label{theorem:cont hardy2}

Let $0<\mathfrak{q}\leq \infty$ and $1 \leq \mathfrak{p} \leq \infty$. Let $u,v\geq 0$. The smallest constant $K$ for which the inequality 

\begin{equation}
\label{eq:conthardyp2}
\left(\int_{0}^\infty u \left(\int_{0}^{x} g \right)^\mathfrak{q} \right)^{\frac{1}{\mathfrak{q}}} \leq K \left(\int_{0}^\infty v g^\mathfrak{p} \right)^{\frac{1}{\mathfrak{p}}}
\end{equation}
holds for any non-negative function $g$ satisfies
\begin{enumerate}[label=(\roman*)]
    \item for $\mathfrak{q}\geq \mathfrak{p}$ 
    \begin{equation}
\label{eq:condiconthardy21}K\approx_{\mathfrak{p},\mathfrak{q}} \sup_{x>0} \left(\int_x ^\infty u \right)^{\frac{1}{\mathfrak{q}}} \left(\int_0 ^x v^{\frac{1}{1-\mathfrak{p}}}\right)^\frac{1}{\mathfrak{p}'};\end{equation}
    \item for $\mathfrak{q}<\mathfrak{p}$ \begin{equation}
\label{eq:condiconthardy2}K\approx_{\mathfrak{p},\mathfrak{q}}\left(\int_{0}^\infty u \left(\int_{x}^\infty u \right)^{\frac{\mathfrak{r}}{\mathfrak{p}}} \left(\int_{0}^x v^{\frac{1}{1-\mathfrak{p}}} \right)^{\frac{\mathfrak{r}}{\mathfrak{p}'}}\right)^{\frac{1}{\mathfrak{r}}}.\end{equation}

\end{enumerate}

\end{lemma}
\begin{lemma}[\protect{\cite[Theorem~6]{kufner2007hardy}}] 
\label{theorem:cont hardy}
Let $0<\mathfrak{q}\leq \infty$ and $1 \leq \mathfrak{p} \leq \infty$. Let $u,v\geq 0$. The smallest constant $K$ for which the inequality

\begin{equation}
\label{eq:conthardyp}
\left(\int_{0}^\infty u \left(\int_{x}^{\infty} g \right)^\mathfrak{q} \right)^{\frac{1}{\mathfrak{q}}} \leq K \left(\int_{0}^\infty v g^\mathfrak{p} \right)^{\frac{1}{\mathfrak{p}}}
\end{equation}
holds for any non-negative function $g$ satisfies

\begin{enumerate}[label=(\roman*)]
\item for $\mathfrak{q}\geq \mathfrak{p}$ 
    \begin{equation}
\label{eq:condiconthardy212}K\approx_{\mathfrak{p},\mathfrak{q}} \sup_{x>0} \left(\int_0 ^x u \right)^{\frac{1}{\mathfrak{q}}} \left(\int_x ^\infty v^{\frac{1}{1-\mathfrak{p}}}\right)^\frac{1}{\mathfrak{p}'};\end{equation}
\item for $\mathfrak{q}<\mathfrak{p}$
\begin{equation}
\label{eq:condiconthardy}K\approx_{\mathfrak{p},\mathfrak{q}} \left(\int_{0}^\infty u \left(\int_{0}^x u \right)^{\frac{\mathfrak{r}}{\mathfrak{p}}} \left(\int_{x}^\infty v^{\frac{1}{1-\mathfrak{p}}} \right)^{\frac{\mathfrak{r}}{\mathfrak{p}'}}\right)^{\frac{1}{\mathfrak{r}}}.\end{equation}
\end{enumerate}
%where $r^{-1}=q^{-1}-p^{-1}$.
\end{lemma}

We will also need several inequalities for monotone functions. The next two results are reduction theorems for non-increasing functions, which transform inequalities for monotone functions into equivalent estimates for non-negative functions.

%Finally, in order to obtain the explicit conditions of \eqref{ineq:FGxy} given in Theorem \ref{theorem:mainh} we need to use the following three more technical results:

\begin{theorem}[\protect{\cite[Theorem~3.1]{stepanov}}]
\label{theorem:gogatoriginal}
Let $1 \leq \mathfrak{p}<\infty$. Let $H$ be a operator on positive functions defined on $\mathbb{R}_+$ 
with the property that
$H(f)\leq H(g)$ for $f\leq g$.
%and $H(f+ \lambda 1) \lesssim H(f)+ \lambda H(1)$, where $1$ is the constant function equal to 1 on $\mathbb{R}_+$.
Let also $b$ be a non-negative function such that 
$\int_0^\infty b=\infty$.
 Then the following are equivalent:
\begin{enumerate}[label=(\roman*)]
    \item for any non-increasing function $f$  the inequality
    \begin{equation}
    H(f)\leq K_1 \left(\int_0^\infty b f^\mathfrak{p}\right)^\frac{1}{\mathfrak{p}}
    \end{equation}
holds;
 \item for any non-negative function $g$ the  inequality
    \begin{equation}
   H(\bar{G}) \leq K_2 \left(\int_0 ^\infty b^{1-\mathfrak{p}} B ^\mathfrak{p}  g^\mathfrak{p} \right)^\frac{1}{\mathfrak{p}},
   \end{equation}
holds, where $\bar{G}(x)=\int_x^\infty g$ and
$B(x)=\int_0^x b$. 
%, and
%\begin{equation}
%    H(1)\leq K_3 B_\infty ^{\frac{1}{p}}.\end{equation}

\end{enumerate}
Moreover, $K_1\approx_{\mathfrak{p}} K_2$.
\end{theorem}

\begin{theorem}[\protect{\cite[Theorem 3.12]{stepanov}}]
\label{th:oinarov}
    Let $0<\mathfrak{q}<\mathfrak{p}\leq 1$. Then the following are equivalent:
    \begin{enumerate}[label=(\roman*)]
    \item for any non-increasing function $f$ the inequality
    \begin{equation}
        \left(\int_0 ^\infty w(x) \left(\int_x ^\infty f\right)^\mathfrak{q} dx \right)^{\frac{1}{\mathfrak{q}}} \leq K_1 \left(\int_0 ^\infty f^\mathfrak{p} \nu \right)^{\frac{1}{\mathfrak{p}}}
    \end{equation} holds;
    \item for any non-negative function $g$ the inequality
    \begin{equation}
        \left(\int_0 ^\infty w(x) \left(\int_x ^\infty (y-x)^\mathfrak{p} g(y) dy \right)^\frac{\mathfrak{q}}{\mathfrak{p}} dx\right)^{\frac{\mathfrak{p}}{\mathfrak{q}}} \leq K^\mathfrak{p}_2 \int_0 ^\infty g(x) \left(\int_0 ^x \nu \right) dx
    \end{equation} holds.
    \end{enumerate}
    Moreover, $K_1 \approx_{\mathfrak{p},\mathfrak{q}} K_2$.
\end{theorem}
The last result of this subsection is a reverse Hardy inequality for monotone functions.
\begin{theorem}[\protect{\cite[Theorem 4.3 and Remark 4.5]{Krepela2022}, see also \cite[Theorem 1.8]{Gogatishvili2006}}]
\label{theorem:krepela}
Let $0<\mathfrak{q}<1$.
Assume that $t \mapsto 1/\nu (t)$ and $t \mapsto \nu (t) /t$ are non-increasing. Set $W(x)=\int_0 ^x w$ and 
\begin{eqnarray*}\label{eq:defxi}
\xi(t)&=&\left(\int_0^\infty  W^\frac{\mathfrak{q}}{1-\mathfrak{q}}(s)w(s)s^{-\frac{\mathfrak{q}}{1-\mathfrak{q}}}\min\{t^\frac{\mathfrak{q}}{1-\mathfrak{q}}, s^\frac{\mathfrak{q}}{1-\mathfrak{q}}\} ds \right)^{1-\mathfrak{q}}\\ &\approx_{\mathfrak{q}}& t^\mathfrak{q} \left(\int_t ^\infty s^{-\frac{1}{1-\mathfrak{q}}} W^{\frac{1}{1-\mathfrak{q}}}(s)ds \right)^{1-\mathfrak{q}}.
  \end{eqnarray*}
  Then the inequality
  \begin{equation}
  \label{eq:krepela}
  \left(\int_0 ^\infty f^{\mathfrak{q}} w \right)^{\frac{1}{\mathfrak{q}}}
      \leq K_1 \sup_{x>0} \frac{ \nu (x) }{x} \int_0 ^x f
  \end{equation}
  holds  for any non-increasing function $f$ if and only if $\xi(t)<\infty$ for any $t>0$ and 
 \begin{equation}
 \label{eq:krep}
K_2:=\left(\int_0^\infty \nu ^{-\mathfrak{q}}(t)\xi^{-\frac{\mathfrak{q}}{1-\mathfrak{q}}}(t)W^\frac{\mathfrak{q}}{1-\mathfrak{q}}(t)w(t) dt\right)^\frac1{\mathfrak{q}}< \infty.
        \end{equation}
        Moreover, $K_1 \approx _\mathfrak{q} K_2$.\end{theorem}
\section{Necessary conditions for  weighted Fourier inequalities}
In this section we study the inequality 
\begin{equation}
\label{ineq:pittnomono}   \big\|u\widehat{f}\,\big\|_{q}\leq C \norm{ f v}_p
\end{equation}
with general weights $u$ and $v$.
This section has three parts: in the first one, we derive a simple necessary condition for
\eqref{ineq:pittnomono} in the case $p>2$; in the second one, we prove that condition \eqref{ineq:55}, shown to be sufficient for \eqref{ineq:pittnomono} when $\max(q,p')\geq 2$ by Benedetto and Heinig, is also necessary whenever $q<p$; finally, in the third one, we obtain an additional necessary condition for the case $q<2<p$. (Recall that it was shown in \cite{sin} that \eqref{ineq:55} is not sufficient for $\eqref{ineq:pittgen}$ in the case $q<2<p$.)

Before proceeding to the new results, 
we recall 
the following well-known necessary condition for  \eqref{ineq:pittnomono}.

\begin{proposition}[\protect{\cite{decarli,h}}]
\label{lemma:case1}
Let $0< q \leq \infty$ and $1 \leq p \leq \infty$. Assume that inequality \eqref{ineq:pittnomono} holds for $f\in L^1$. Then
\begin{equation}
\label{cond:intervals}
    \sup_{%A,B \text{ balls, }
     |A||B|= 1} \left(\int_A u^{q} \right)^{\frac{1}{{q}}} \left( \int_B v^{-p'} \right)^{\frac{1}{p'}} \lesssim_{p,q} C,
\end{equation}  where
the supremum is taken over all cubes $A$ and $B$ in $\R^d$ and $|A|$ is the sidelength of $A$. In particular, both $u^q$ and $v^{-p'}$ are locally integrable.

\end{proposition}

\label{section:necessary}
\subsection{The case \texorpdfstring{$p>2$}{p>2}}
The next result improves Proposition 
\ref{lemma:case1}
for
$p>2$.
%some preliminary necessary conditions for \eqref{ineq:pittgen} to hold.
\begin{lemma}
\label{lemma:prelinec}
    Let $2< p\leq \infty$ and $0<q\leq \infty$. Assume that inequality \eqref{ineq:pittnomono} holds for $f\in L^1$. Then
\begin{equation}\label{new-n}
 \sup_{\substack{s>0\\|A|=1/s}} \left( \int_{A} u^q  \right)^{\frac{1}{q}}   \left(\sum_{n \in \mathbb{Z}^d} \left(\int_{sn+ s[-\frac{1}{2}, \frac{1}{2}]^d}v^{-p'}\right) ^{\frac{p^{\sharp}}{p'}} \right) ^{\frac{1}{p^{\sharp}}} \lesssim_{p,q} C,
 \end{equation}
 where the supremum is taken over all cubes $A$ in $\R^d$.

\end{lemma}

\begin{proof}
Note that if $q=\infty$ it is even possible to determine  the exact value of the constant $C$ in inequality \eqref{ineq:pittnomono}, see
 %is trivial and it is discussed in 
 Proposition \ref{theorem:degheinig}. 
 Assume now $q<\infty$.
   We exploit the behaviour of the Fourier transform with respect to translations (cf.    \cite[Theorem 3.1]{Alg}).

    Fix $s>0$. Let $M\in \mathbb{N}$  and define \begin{equation*}f(x)=v^{-p'}(x)  \sum_{\substack{n\in \mathbb{Z}^d\\|n|\leq M}} \varepsilon _n \lambda_n \mathbbm{1}_{[0,s]}(|x-2ns|)=:\sum_{\substack{n\in \mathbb{Z}^d\\|n|\leq M}}  \varepsilon _n \lambda_n f_n(x) \in L^1(\mathbb{R}^d), \end{equation*} where $\lambda_n \in \mathbb{R}$ and $\varepsilon_n=\pm 1$. Clearly,
\begin{equation*}\norm{fv}_{p} =\Bigg(
\sum_{{\substack{n\in \mathbb{Z}^d\\|n|\leq M}} }
|\lambda_n|^p  \int_ {|y-2ns|<s} v^{-p'}(y) dy\Bigg)^{\frac{1}{p}}=:\Bigg(\sum_{{\substack{n\in \mathbb{Z}^d\\|n|\leq M}} } |\lambda_n|^p  V_n\Bigg)^{\frac{1}{p}}\end{equation*}
with the usual modification for $p=\infty$
and\begin{equation*}\widehat{f}(\xi) = \sum_{{\substack{n\in \mathbb{Z}^d\\|n|\leq M}} } \varepsilon _n \lambda_n \widehat{f} _n(\xi).\end{equation*}
Hence, using the Khinchin inequality \eqref{khi}, we have that

\begin{eqnarray*}
\mathbb{E}\left [ \norm{\widehat{f}u}_{q}^q\right] &\approx_q& \int_{\mathbb{R}^d} u^q(\xi) \Bigg(\sum_{ {\substack{n\in \mathbb{Z}^d\\|n|\leq M}} } |\lambda _n |^2 |\widehat{f}_n(\xi)|^2\Bigg)^{\frac{q}{2}} d\xi \\&\geq& \int_{|\xi|< \frac{1}{2 \pi s}} u^q(\xi) \Bigg(\sum_{ \substack{n\in \mathbb{Z}^d\\|n|\leq M}} |\lambda _n |^2 |\widehat{f}_n(\xi)|^2\Bigg)^{\frac{q}{2}} d\xi .
\end{eqnarray*}
We note that if $|\xi|< \frac{1}{2 \pi s}$,

\begin{eqnarray*}
|\widehat{f}_n(\xi)| &=&
 \left |\int_{|y-2ns|<s} v^{-p'}(y) e^{2 \pi i \langle \xi, y \rangle} dy\right|\\&=& \left |\int_{|y|<s} v^{-p'}(y+2ns) e^{2 \pi i \langle \xi, y \rangle} dy\right| \gtrsim \int_{|y-2ns|<s} v^{-p'}=V_n.
\end{eqnarray*}
Thus, since
inequality \eqref{ineq:pittnomono} holds, by letting $M\to \infty$ we obtain
\begin{equation*}
\left( \int_{|\xi|\leq \frac{1}{2 \pi s}} u^q \right)^{\frac{1}{q}} \left(\sum_{ n\in \mathbb{Z}^d} |\lambda _n |^2 V_n^2 \right) ^{\frac{1}{2}} \lesssim_{p,q} C \left(\sum_{ n\in \mathbb{Z}^d} |\lambda _n |^p V_n \right) ^{\frac{1}{p}}.
\end{equation*}
Choosing $\lambda _n = V_n^{1/(p-2)}$,
 %from H\"{o}lder's inequality
   we deduce that
\begin{equation*}
\left( \int_{|\xi|\leq \frac{1}{2 \pi s}} u^q \right)^{\frac{1}{q}}   \left(\sum_{ n\in \mathbb{Z}^d}  V_n ^{\frac{p^\sharp}{p'}} \right) ^{\frac{1}{p^{\sharp}}} \lesssim_{p,q} C.
\end{equation*}
It remains to note that by translating $\widehat{f}$ we can replace $\int_{|\xi|\leq \frac{1}{2 \pi s}} u^q $ by $\int_{|\xi-\xi_0|\leq \frac{1}{2 \pi s}} u^q$ for any $\xi_0 \in \R^d$ and that any cube can be covered by a finite number of balls. 
\end{proof}

 For the sake of completeness, we note that \eqref{cond:intervals} can be obtained following the proof of Lemma \ref{lemma:prelinec} with
 $\lambda_n=1$ for $n=(0, \dots, 0)$ and $\lambda_n=0$ otherwise.

\subsection{The case \texorpdfstring{$q<p$}{q<p}}

We are going to sharpen 
the argument which yields condition \eqref{cond:intervals} by considering a randomized superposition 
of the functions we used in the proof of
Lemma \ref{lemma:prelinec}
for different cube sizes.

%When $q<p$, the argument which proves condition \eqref{cond:intervals} can be improved by considering a randomized superposition of 
%the functions which yield an  \eqref{cond:intervals} for different cube sizes. 
%This is similar to a standard proof of the Hardy inequality (see for instance Theorem 2 in \cite{prokhorov}).

\begin{theorem}\label{lemma:necesnew}

Let $1\leq p \leq \infty$, $0<q<\infty$, and $q<p$. Assume that inequality \eqref{ineq:pittnomono} holds for $f\in L^1$.
%with constant $C$.
 Then, setting

\begin{equation*}\mathbf{V}(t) =\int_{\mathbb{S}^{d-1}} v^{-p'}(tx) dx\qquad\mbox{
and}\qquad
\mathbf{U}(t) =\int_{\mathbb{S}^{d-1}} u^q(tx) dx,\end{equation*} we have

 \begin{equation}\label{5}\left(\int_{0}^\infty \mathbf{U}(t) t^{d-1}  \left(\int_0 ^t \mathbf{U}(s) s^{d-1} ds\right)^{\frac{r}{p}} \left(\int_0 ^{\frac{1}{2 \pi t}} \mathbf{V}(s) s^{d-1} ds\right)^{\frac{r}{p'}} dt\right)^{\frac{1}{r}}\lesssim_{p,q} C.
 \end{equation} In particular, if $u$ and $v^{-1}$ are radial, non-increasing,
  \begin{equation}
  \label{ineq:55}
      % K\approx
      \left(
       \int_0^\infty u^{*,q}(s)
       \left(\int_{0}^{s} u^{*,q}\right)^{\frac{r}{p}}\left(\int_{0}^{1/s} v_*^{-p'}\right)^{\frac{r}{p'}}  ds \right)^{\frac{1}{r}}\lesssim_{p,q} C.
    \end{equation}
 \begin{comment}
or,
equivalently,
 \begin{equation}\label{6}
 \left(\int_{0}^\infty V(t) t^{d-1}  \left(\int_0 ^t V(s) s^{d-1} ds\right)^{\frac{r}{q'}} \left(\int_0 ^{\frac{1}{2 \pi t}} \mathbf{U}(s) s^{d-1} ds\right)^{\frac{r}{q}} dt\right)^{\frac{1}{r}}\lesssim C.
  \end{equation}
%In particular, for radially non-increasing $u,v^{-1}$ and $q<p$, $\max(q,p')\geq 2$, equation \eqref{nec1} is sufficient and necessary.
\end{comment}
\end{theorem}

\begin{proof}
If $p=1$ it is possible to obtain the exact value of the constant $C$ in inequality \eqref{ineq:pittnomono}, see
 %is trivial and it is discussed in 
 Proposition \ref{theorem:degheinig}. 
 
Let $p>1$.
Fix $M>0$.
By  \eqref{cond:intervals}, we know that $v^{-p'}$ is locally integrable and, hence, \begin{equation*}I:=\int_0 ^M \mathbf{V}(t) t^{d-1} dt<\infty.\end{equation*} Define the decreasing sequence $(\alpha _n)_{n\geq 0}$ so that it satisfies the following properties:
\begin{enumerate}[label=(\roman*)]
    \item  $\alpha_{-1}= \infty$, $\alpha_0=M;$
    \item for $n\geq 1$, \begin{equation*}W_n:=\int_{\alpha_{n}}^{\alpha_{n-1}} \mathbf{V}(t) t^{d-1} dt=\frac{I}{2^n}.\end{equation*}
\end{enumerate}
Let $\lambda_n\in \mathbb{R}_+$ be an arbitrary sequence with finitely non-zero terms. Consider the sequence $(f_n)_{n\geq 1}$, where $f_n(x)= \lambda_n v^{-p'}(x) \mathbbm{1}_{[\alpha_{n}, \alpha_{n-1}]}(|x|)$.  
Let \begin{equation*}f=\sum_{n=1}^\infty \varepsilon _n f_n,\end{equation*} where $\varepsilon_n=\pm 1$.
Then, for this $f$  inequality \eqref{ineq:pittnomono} yields 
\begin{align*}\int_{\mathbb{R}^d} u^q(\xi) \left|\sum_{n=1}^\infty \varepsilon_n \widehat{f}_n (\xi)\right|^q d \xi
&\leq C^q \norm{fv}_p^q
\\ &
%\qquad\qquad\quad
= C^q \left(\sum_{n=1}^\infty  \lambda _n^p \int_{\alpha_{n}}^{\alpha_{n-1}} \mathbf{V}(t) t^{d-1} dt \right)^{\frac{q}{p}}\\ &= C^q \left(\sum_{n=1}^\infty \lambda _n ^p W_n  \right)^{\frac{q}{p}}.\end{align*}
Here taking expected values of both sides and using  the Khinchin inequality \eqref{khi}, we deduce that
\begin{equation}
\label{eq:mothereq}
\int_{\mathbb{R}^d} u^q(\xi) \left|\sum_{n=1}^\infty |\widehat{f} _n (\xi)|^2 \right|^\frac{q}2 d \xi
\lesssim_q C^q \left(\sum_{n=1}^\infty \lambda _n ^p W_n  \right)^{\frac{q}{p}}.
\end{equation} 
Furthermore, 
  for $0\leq |\xi| \leq \frac{1}{2 \pi \alpha_{n-1}}$ we have
\begin{eqnarray*}
|\widehat{f_n}(\xi)|
&=& \Big|\lambda_n\int_{\alpha_{n}\leq |x|\leq{\alpha_{n-1}}}v^{-p'}(x) e^{-2\pi i \langle x,\xi \rangle }dx\Big|
\\
&\gtrsim&
  \lambda_n \int_{\alpha_{n}\leq |x|\leq{\alpha_{n-1}}} v^{-p'}(x) dx=\lambda_n W_n.
\end{eqnarray*} Observing that
\begin{align*}
    \int_{\mathbb{R}^d} u^q(\xi) \left|\sum_{n=1}^\infty |\widehat{f} _n (\xi)|^2 \right|^\frac{q}2 d \xi  %\\&\qquad\qquad
    &\geq \sum_{m=-1}^\infty \int_{\frac{1}{2 \pi \alpha _{m}} \leq |\xi | \leq \frac{1}{2 \pi \alpha _{m+1}}} u^q(\xi) \left|\sum_{n=m+2}^\infty |\widehat{f} _n (\xi)|^2 \right|^\frac{q}2 d \xi 
     \\&%\qquad\qquad
     \gtrsim \sum_{m=-1}^\infty \int_{\frac{1}{2 \pi \alpha _{m}} \leq |\xi | \leq \frac{1}{2 \pi \alpha _{m+1}}} u^q(\xi)  d \xi  \left|\sum_{n=m+2}^\infty \lambda_n^2  W_n^2  \right|^\frac{q}2,
\end{align*}
we conclude that the inequality
\begin{equation*}\left(\sum_{m=-1}^\infty \int_{\frac{1}{2 \pi \alpha _{m}} \leq |\xi | \leq \frac{1}{2 \pi \alpha _{m+1}}} u^q(\xi) d \xi \left(\sum_{n=m+2}^{\infty} \lambda_n^2 W_n ^2 \right)^{\frac{q}{2}}\right)^{\frac{1}{q}} \lesssim_q C \left(\sum_{n=1}^\infty \lambda _n ^p W_n \right)^{\frac{1}{p}}\end{equation*}
 holds for any non-negative $\lambda_n$.

 Equivalently, for any  non-negative $\lambda_n$,
we have \begin{equation*}\left(\sum_{m=1}^\infty \int_{\frac{1}{2 \pi \alpha _{m-2}} \leq |\xi | \leq \frac{1}{2 \pi \alpha _{m-1}}} u^q(\xi) d \xi \left(\sum_{n=m}^{\infty} \lambda_j \right)^{\frac{q}{2}}\right)^{\frac{2}{q}} \lesssim_q C^2\left(\sum_{n=1}^\infty \lambda _n ^{\frac{p}{2}} W_n^{1-p} \right)^{\frac{2}{p}}.\end{equation*}
The last inequality can be seen to imply condition \eqref{5} through routine arguments.

First, from the characterization of the Hardy inequalities (Lemma \ref{theorem:discretehardy} with $\mathfrak{q}=\frac{q}{2}$ and $\mathfrak{p}=\frac{p}{2}$), we deduce the following conditions:
%\begin{enumerate}[label=(\roman*)]
 %   \item[$\cdot$]
 \\
  if $1\leq p\leq 2$,
   \begin{align*} &\left(\sum_{m=1}^\infty\left( \int_{0 \leq |\xi|\leq \frac{1}{2 \pi \alpha_{m-1}}} u^q(\xi) d \xi \right) ^{\frac{r}{p}} \left( \inf_{n \geq m}W_n ^{1-p}\right)^{-\frac{r}{p}} 
  \right. \\
&\qquad\qquad\qquad\qquad\qquad\qquad\quad\left.
  \times
   \int_{\frac{1}{2 \pi \alpha _{m-2}} \leq |\xi | \leq \frac{1}{2 \pi \alpha _{m-1}}} u^q(\xi) d \xi \right)^{\frac{2}{r}} \lesssim_{p,q} C^2;\end{align*}
  if $p>2$,
   \begin{align*} &
   \left(\sum_{m=1}^\infty\left( \int_{0 \leq |\xi|\leq \frac{1}{2 \pi \alpha_{m-1}}} u^q(\xi) d \xi \right) ^{\frac{r}{p}} \left(\sum_{n=m}^\infty W_n^{\frac{1-p}{1-p/2}}\right) ^{\frac{r}{2}-\frac{r}{p}}
   \right. \\
&\qquad\qquad\qquad\qquad\qquad\qquad\quad\left.
  \times
   \int_{\frac{1}{2 \pi \alpha _{m-2}} \leq |\xi | \leq \frac{1}{2 \pi \alpha _{m-1}}} u^q(\xi) d \xi \right)^{\frac{2}{r}} \lesssim_{p,q} C^2.\end{align*}
   %\end{enumerate}

Second, observe that $W_n=I 2^{-n}$. Hence $\left(\inf_{n\geq m} W_n^{1-p}\right)^{-\frac{r}{p}}=W_m^{\frac{r}{p'}}$ and, for $p>2$, \begin{equation*}\left(\sum_{m=m}^\infty W_n^{\frac{1-p}{1-p/2}}\right) ^{\frac{r}{2}-\frac{r}{p}} \approx_{p,q} W_m^{\frac{r}{p'}} \approx_{p,q} \left(\int_{0}^{\alpha_{m-1}}\mathbf{V}(t) t^{d-1} dt\right)^{\frac{r}{p'}}.\end{equation*} Thus, both previous expressions are equivalent to
\begin{align*} & \left(\sum_{m=1}^\infty\left( \int_{0 \leq |\xi|\leq \frac{1}{2 \pi \alpha_{m-1}}} u^q(\xi) d \xi \right) ^{\frac{r}{p}} \left(\int_0 ^ {\alpha_{m-1}}\mathbf{V}(t) t^{d-1} dt \right) ^{\frac{r}{p'}}
\right. \\
&\qquad\qquad\qquad\qquad\qquad\qquad\quad\left.
  \times
\int_{\frac{1}{2 \pi \alpha _{m-2}} \leq |\xi | \leq \frac{1}{2 \pi \alpha _{m-1}}} u^q(\xi) d \xi \right)^{\frac{1}{r}} \lesssim_{p,q} C,\end{align*}
which is the same as 
\begin{align*} & \left(\sum_{m=1}^\infty \int_{\frac{1}{2 \pi \alpha _{m-2}}}^{\frac{1}{2 \pi \alpha _{m-1}}} \mathbf{U}(t) t^{d-1}d t \left( \int_{0}^{\frac{1}{2 \pi \alpha_{m-1}}} \mathbf{U}(t) t^{d-1} d t \right) ^{\frac{r}{p}} 
\right. \\
&\qquad\qquad\qquad\qquad\qquad\qquad\quad\left.
  \times
\left(\int_0 ^ {\alpha_{m-1}}\mathbf{V}(t) t^{d-1} dt \right) ^{\frac{r}{p'}} \right)^{\frac{1}{r}} \lesssim_{p,q} C.\end{align*}

Finally, by monotonicity, we conclude that

\begin{align*}&\left(\int_{\frac{1}{2 \pi M}}^\infty \mathbf{U}(t) t^{d-1}  \left(\int_0 ^t \mathbf{U}(s) s^{d-1} ds\right)^{\frac{r}{p}} \left(\int_0 ^{\frac{1}{2 \pi t}} \mathbf{V}(s) s^{d-1} ds\right)^{\frac{r}{p'}} dt\right)^{\frac{1}{r}} \\
&\leq  \left(\sum_{m=2}^\infty \int_{\frac{1}{2 \pi \alpha _{m-2}}}^{\frac{1}{2 \pi \alpha _{m-1}}} \mathbf{U}(t) t^{d-1} dt\left( \int_{0}^{\frac{1}{2 \pi \alpha_{m-1}}} \mathbf{U}(t) t^{d-1} d t \right) ^{\frac{r}{p}} \left(\int_0 ^ {\alpha_{m-2}}\mathbf{V}(t) t^{d-1} dt \right) ^{\frac{r}{p'}} \right)^{\frac{1}{r}} \\
&\lesssim_{p,q}  \left(\sum_{m=2}^\infty \int_{\frac{1}{2 \pi \alpha _{m-2}}}^{\frac{1}{2 \pi \alpha _{m-1}}} \mathbf{U}(t) t^{d-1}dt \left( \int_{0}^{\frac{1}{2 \pi \alpha_{m-1}}} \mathbf{U}(t) t^{d-1} d t \right) ^{\frac{r}{p}} \left(\int_0 ^ {\alpha_{m-1}}\mathbf{V}(t) t^{d-1} dt \right) ^{\frac{r}{p'}} \right)^{\frac{1}{r}} \\
&\lesssim_{p,q}  C,\end{align*}
where we used that, for $m\geq 2$,
\begin{equation*} \int_0 ^ {\alpha_{m-1}}\mathbf{V}(t) t^{d-1} dt\approx \int_0 ^ {\alpha_{m-2}}\mathbf{V}(t) t^{d-1} dt.\end{equation*} The result follows by letting $M \to \infty$.
\end{proof}
\begin{remark}
\label{remark:K}

    Note that in the proof of Theorem \ref{lemma:necesnew} the only needed property of the Fourier transform is $$|\widehat{f}(\xi)| \gtrsim \int_{|x|\le s} f(x)dx\quad  \quad{for}\quad 
    |\xi|\lesssim \frac{1}{2\pi s}$$ for   any non-negative function $f$ supported on $[0,s]$. Thus, the result also holds true for any integral transform  $Tf (\xi)=\int_{\mathbb{R}^d} K(\langle \xi,x \rangle ) f(x) dx$ with $K(y)\gtrsim 1$ for $|y|\leq 1$.
See \cite{indiana} for  further  examples of such transforms.

\end{remark}
Combining Theorem \ref{lemma:necesnew} and Proposition \ref{lemma:case1} with the characterization of the Hardy inequality (Lemma \ref{theorem:cont hardy2}), we deduce 
\begin{corollary}\label{coro:FH}
   Let
    $1\leq p \leq \infty$ and $0<q<\infty$. Then
   inequality \eqref{ineq:pittgen} with radial, monotone weights always implies the Hardy type inequality
    $$\norm{ u^* \int_0 ^{1/t} f}_q \lesssim \norm{fv_*}_p.$$
\end{corollary}

\subsection{The case \texorpdfstring{$q<2<p$}{q<2<p}}
In the Banach space regime, that is, for $1\leq p,q \leq \infty$, a duality argument (see Lemma \ref{lemma:duality}) shows that the validity of inequality \eqref{ineq:pittgen} is preserved under the transformation $(u,v,q,p) \mapsto (v^{-1},u^{-1},p',q')$. 
While it is easy to see that both \eqref{cond:intervals} and \eqref{ineq:55} do not change under the aforementioned transformation, this is clearly not the case for  condition \eqref{new-n}.
%However, in condition \eqref{new-n} it is clear that $u$ and $v$ play a completely different role.
% It is therefore natural to look for a symmetric version of \eqref{new-n} valid in the case $1\leq q<2<p\leq \infty$.

In this section we obtain a symmetric version of \eqref{new-n} for the case
$0< q<2$ and $2<p\leq \infty$. Unlike the proofs of Lemma \ref{lemma:prelinec} and Theorem \ref{lemma:necesnew}, in this case the proof is
 non-constructive 
 and relies on the factorization result given in Theorem \ref{coro:usedfactor}.

\begin{lemma}\label{prop:nopacific}Assume that inequality \eqref{ineq:pittnomono} holds with $0<q<2<p\leq \infty $  for any $f\in L^1$. Then,
\begin{enumerate}[label=(\roman*)]

\item 
We have
 \begin{equation}\label{symmetric}
 \sup _{t>0}  \left(\sum_{ n\in \mathbb{Z}^d}  \left(\int_{\frac n{t}+ \frac1{t}[-\frac{1}{2}, \frac{1}{2}]^d} v^{-p'}\right) ^{\frac{p^{\sharp}}{p'}} \right) ^{\frac{1}{p^{\sharp}}}  \left(\sum_{ n\in \mathbb{Z}^d}  \left(\int_{nt+ t[-\frac{1}{2}, \frac{1}{2}]^d} u^q\right) ^{\frac{q^{\sharp}}{q}} \right) ^{\frac{1}{q^{\sharp}}} \lesssim_{p,q} C.
 \end{equation}
   
%As a consequence,  we obtain that

\item For a fixed $t>0$,
assume that  $u$ is radial non-increasing for $|x| \geq t/2$ and $v^{-1}$ is radial non-increasing for $|x| \geq 1/2t$; then %    we have 
\begin{equation}
\label{eq:nopac2}  
\left(\int_{1/t}^\infty v_0^{-p^{\sharp}}(y) y^{d-1} dy\right)^{\frac{1}{p^{\sharp}}}\left(\int_{t}^\infty u_0^{q^{\sharp}}(\xi) \xi^{d-1} d\xi \right)^{\frac{1}{q^{\sharp}}}\lesssim_{p,q} C.  
\end{equation}
\end{enumerate}
   %\begin{equation*}\sup _t (\int_{\frac{1}{2t}}^\infty v(y)^{-R} y^{d-1}  dy )^{\frac{1}{p^{\sharp}}} (\int_{\frac{t}{2}}^\infty u(y)^{q^{\sharp}} y^{d-1} dy)^{\frac{1}{q^{\sharp}}}\lesssim C.\end{equation*}
\end{lemma}

\begin{proof}
(i) % Let $X$ be the closure of the bounded and compactly supported functions with respect to $\norm{\cdot}_{p,v}$. Clearly, $X$ is Banach lattice and the Fourier transform maps $X$ into the Banach lattice $L^p_v$.
By replacing $v$ by $ v_{\epsilon,N}(x):= \max(\varepsilon, v(x)) \mathbbm{1}^{-1}_{[0,N]}(|x|)$, we see that the Fourier transform maps $L^p({v_{\varepsilon,N}})\subset L ^1$ into $L^q(u)$. 
Thus, by applying Theorem \ref{coro:usedfactor}, we find
 $h$ with $\norm{h^{-1}}_{q^\sharp}=1$ such that for any $f\in L^1$
 \begin{equation}
 \label{factor-aux}
 \norm{\widehat{f} u h}_2 \lesssim_q C \norm{fv_{\varepsilon,N}}_p.\end{equation}
    From $\norm{h^{-1}}_{q^\sharp}=1$, by Hölder's inequality, we obtain
    \begin{equation*}\left(\sum_{ n\in \mathbb{Z}^d}  \left(\int_{nt+ t[-\frac{1}{2}, \frac{1}{2}]^d} u^q\right) ^{\frac{q^{\sharp}}{q}} \right) ^{\frac{1}{q^{\sharp}}} \leq   \sup_{ |A|=t} \left( \int_{A} u^2  h^2  \right)^{\frac{1}{2}},\end{equation*}
    where the supremum is taken over all
    cubes
    $A$ with sidelength 
    ${|A|=t}$.
    Next, applying  Lemma \ref{lemma:prelinec} to
    inequality 
    \eqref{factor-aux}, we get
    \begin{equation*}\sup_{t>0, |A|=t} \left( \int_{A} u^2  h^2\right)^{\frac{1}{2}}   \left(\sum_{ n\in \mathbb{Z}^d}  \left(\int_{\frac n{t}+ \frac1{t}[-\frac{1}{2}, \frac{1}{2}]^d} v^{-p'}_{\varepsilon,N}\right) ^{\frac{p^{\sharp}}{p'}} \right) ^{\frac{1}{p^{\sharp}}}\lesssim_{p,q} C.
    \end{equation*} 
       Letting $\varepsilon \to 0$ and $N \to \infty$
       completes 
       the  proof of part (i).
        
To prove (ii), by a scaling argument, it suffices to see that 
\eqref{symmetric} implies \eqref{eq:nopac2}
 for $t=1$.
    Observe that for $k\geq 1$, by monotonicity,\begin{eqnarray*}\sum_{ \norm{n}_\infty =k}   \left(\int_{n+ [-\frac{1}{2}, \frac{1}{2}]^d} u^q \right) ^{\frac{q^{\sharp}}{q}}  &\gtrsim_q& k^{d-1} u_0^{q^{\sharp}} (\sqrt{d}(k+1/2)) \\&\gtrsim_q& \int_{\sqrt{d}(k+1/2) }^ {\sqrt{d}(k+3/2) } u_0^{q^{\sharp}}(\xi) \xi^{d-1} dy.
    \end{eqnarray*}
    Moreover, for     $k=1$,
    \begin{equation*}
    \sum_{ \norm{n}_\infty =k}   \left(\int_{n+ [-\frac{1}{2}, \frac{1}{2}]^d} u^q\right) ^{\frac{q^{\sharp}}{q}}  \gtrsim_q  u_0^{q^{\sharp}} (1) \gtrsim \int_{1 }^ {3\sqrt{d}/2 } u_0^{q^{\sharp}}(\xi) \xi^{d-1} d\xi.
    \end{equation*}
   A similar argument for $v$ completes the proof.
    \end{proof}

\section{Proof of Theorem \ref{theorem:mainabstract}}
\label{section:mainproofs}
In this section we give the proof of our main result.
First, using Theorem \ref{lemma:necesnew} and Lemma \ref{prop:nopacific}, we prove the implication $(i)\implies (ii)$ in Theorem \ref{theorem:mainabstract} for $Y=L^\infty(v)$ and $X=L^1(u^q)$. Second, 
 %using standard arguments,
 we extend this result to the general case, completing the proof of Theorem \ref{theorem:mainabstract}.
%  by standard arguments, we extend the previous result to the general case. 
 
\begin{proof}[Proof of $(i)\implies (ii)$ in Theorem \ref{theorem:mainabstract} 
for a special case]
$ $\newline
We consider the case $\norm{ |F|^\beta}^{\frac{1}{\beta}}_X = \norm{Fu}_q$ and  $\norm{G}_Y = \norm{Gv}_\infty$ for $u,v^{-1}$ radial non-increasing and $0<q\leq 1$. Note that in this case inequality \eqref{ineq:pittXY} with $\beta=q$ becomes \eqref{ineq:pittgen}.    
Our goal is to show that if the inequality 
\begin{equation}
    \label{ineq:pittgen++}
\big\|u\widehat{f}\,\big\|
_{q}\lesssim \norm{ f v}_{\infty}, \end{equation}
%\eqref{ineq:pittgen}
holds, then $$\norm{F^\circ u}_q \lesssim_{q} C \norm{G^\circ v}_ \infty,
\quad\mbox{ for any}
\quad F\prec G.
$$
By Corollary \ref{coro:method}, this yields inequality \eqref{ineq:Tfxy}.

Let $F\prec G$ and $\norm{G^\circ v}_\infty\leq1$. We have to show that $\norm{F^\circ u}_q \lesssim_{q} C$ or, equivalently, \begin{equation}
\label{eq:step1}
    \norm{F^* u^*}_q \lesssim_{q} C.
\end{equation} Observe that it suffices to obtain this for $F\prec G=v^{-1}$.

To begin with, note that by Lemma \ref{prop:nopacific} applied to \eqref{ineq:pittgen++}, we have  $\int_1^\infty v^{-2}_0(y) y^{d-1} dy<\infty$.   
Thus, there exists a decreasing sequence $0\leq y_n<
y_{n-1}$ which satisfies 
\begin{equation*} \int_{y_n ^d}^{ y_{n-1}^d} v_*^{-2}(y) dy \approx \int_{  y_{n}}^{   y_{n-1}} v_0^{-2}(y) y^{d-1} dy \approx   2^n,\end{equation*} for  any integer $n \in 
(-\infty, + \infty )$ if $\int_0^\infty v^{-2}_0(y)y^{d-1} dy = \infty$, and for any integer $n \in (-\infty,N]$ if $\int_0^\infty v^{-2}_0(y)y^{d-1} dy <\infty$. In the latter case we also impose that $y_N=0$.

Set $$x_n :=\frac{1}{y_n}.$$ 
%We now split the sets $\mathbb{Z}$ and $\mathbb{Z} \cap (-\infty, N)$ into two disjoint parts, corresponding to the former and latter cases, respectively.
In the case $\int_0^\infty v^{-2}_0(y)y^{d-1} dy = \infty$
we  split the set $\mathbb{Z}$ into two disjoint parts
${I}\cup {J}$ and, similarly, we split the set
 $\mathbb{Z} \cap (-\infty, N)$ 
 if the integral is finite.

If $y_n >2 y_{n+1}$, (equivalently, $2x_n<x_{n+1}$) we say that $n \in I$; otherwise, we say that $n \in J$.

\textbf{Step 1.} We claim that inequality \eqref{ineq:pittgen++} implies 
\begin{equation}
\label{ineq:necstep1}
    \left(\int_0^\infty  u^{*,q} \left( \int_0 ^{x^{-1}} v_*^{-1}\right)^q \right)^{\frac{1}{q}}\lesssim_{q} C
\end{equation} and
\begin{equation}
\label{ineq:necstep2}
\left(\sum_{n \in I} \left(\int_{ x_n^d}^{x_{n+1}^d}u^{*,q^{\sharp}} \right)^{\frac{q}{q^{\sharp}}} 2^{\frac{qn}{2}}\right)^{\frac{1}{q}} \lesssim_{q} C.
\end{equation}
%\sum_{n\in I} \left(\int_{ x_n^d}^{x_{n+1}^d}  u^{*,q^{\sharp}} \right)^{\frac{q}{q^{\sharp}}}  2^{\frac{nq}{2}} \lesssim_{q} C^q,
Since
\eqref{ineq:necstep1}
coincides with
\eqref{ineq:55}, 
by Theorem \ref{lemma:necesnew}, inequality \eqref{ineq:pittgen++} implies \eqref{ineq:necstep1}. To obtain
\begin{equation*}
%\label{ineq:necstep22}
\left(\sum_{n \in I} \left(\int_{  x_n }^{ x_{n+1}} u_0^{q^{\sharp}}(y) y^{d-1} dy \right)^{\frac{q}{q^{\sharp}}} 2^{\frac{qn}{2}}\right)^{\frac{1}{q}} \lesssim_{q} C,
\end{equation*}
which is equivalent to
\eqref{ineq:necstep2},
 we proceed in the following way for each $n\in I$: \newline
Define 
\begin{equation*}A(r_0,r_1):=\{ z \in \mathbb{R}^d: r_0 < |z| \leq r_1\} \end{equation*}
\begin{comment}
Next, observe that if inequality \eqref{ineq:pittgen} holds, then so does the inequality \begin{equation}\label{vspom}\norm{\widehat{f}u \mathbbm{1}_{A(x_n/2,x_{n+1})}}_q \leq C \norm{f v \mathbbm{1}^{-1}_{A(y_{n}/2,y_{n-1})}}_{\infty}. \end{equation} 
\end{comment}
and consider the obvious  inequality
\begin{equation}\label{vspom}\norm{\widehat{f}u \mathbbm{1}_{A(x_n/2,x_{n+1})}}_q \leq D \norm{f v \mathbbm{1}^{-1}_{A(y_{n}/2,y_{n-1})}}_{\infty},\end{equation} where
\begin{equation}
    D:=\sup_ {\norm{g v \mathbbm{1}^{-1}_{A(y_{n}/2,y_{n-1})}}_{\infty}=1} \norm{\widehat{g}u \mathbbm{1}_{A(x_n/2,x_{n+1})}}_q.
\end{equation}
We note that the function $u \mathbbm{1}_{A(x_n/2,x_{n+1})}$ is radial non-increasing for $|x| \geq \frac{x_n}{2}$ and $v \mathbbm{1}^{-1}_{A(y_{n}/2,y_{n-1})}$ is radial non-decreasing for $|y| \geq \frac{1}{2x_n}= \frac{y_n}{2}$. Thus, by item (ii) in Lemma \ref{prop:nopacific} with $t=x_n$, 
 inequality  \eqref{vspom}
implies
$$\left(\int_{  x_n }^{ x_{n+1}} u_0^{q^{\sharp}}(\xi)\xi^{d-1} d \xi \right)^{\frac{1}{q^{\sharp}}} \left(\int_{  y_n }^{  y_{n-1}} v_0^{-2}(y)y^{d-1} d y \right)^{\frac{1}{2}}\lesssim_q D.$$
Hence there exists a function $f_n$ supported on $A(y_{n}/2 ,y_{n-1})$ such that $\norm{f_n v}_{\infty}=1$ and 

\begin{equation}     \label{vspom1}
\begin{split}%\nonumber
    \left(\int_{  x_n }^{ x_{n+1}} u_0^{q^{\sharp}}(\xi)\xi^{d-1} d \xi \right)^{\frac{1}{q^{\sharp}}} 2^{\frac{n}{2}}&\approx \left(\int_{  x_n }^{ x_{n+1}} u_0^{q^{\sharp}}(\xi)\xi^{d-1} d \xi \right)^{\frac{1}{q^{\sharp}}} \left(\int_{  y_n }^{  y_{n-1}} v_0^{-2}(y)y^{d-1} d y \right)^{\frac{1}{2}}\\
&\lesssim_{q} \norm{\widehat{f_n}u \mathbbm{1}_{A(x_{n}/2,x_{n+1})}}_q.
\end{split}\end{equation}

For a fixed $M\in \mathbb{N}$, define %which will later tend to infinity, 

$$f(x)=\sum_{\substack{n\in I \\
   |n|\le M }} \varepsilon_n f_n (x),\qquad \varepsilon_n=\pm 1.  %\mathbbm{1}(n)_{[-N,N]},
$$ 
We claim  that 
%$\sum_{n \in I} \mathbbm{1}_{A(y_{n}/2,y_{n-1})} \leq 2$ or, equivalently,
 for each point $x$ there are at most two different $n\in I$ for which  $f_n(x)\neq 0$.  Indeed, if $n\in I$, $f_n$ is supported on $A(y_n/2, y_{n-1})$ and $y_n > 2 y_{n+1}$. Thus, 
 for $k\geq 2$,
 if $n+k\in I$, the support of $f_{n+k}$  is contained in \begin{equation*}A(0, y_{n+k-1}) \subset A(0, y_{n+1}) ,\end{equation*} and $A(0, y_{n+1})$
is disjoint from $A(y_n/2, y_{n-1})$. Hence, we deduce that $$\norm{fv}_\infty  \leq 2.$$
By Khinchin's inequality \eqref{khi}, there exists a sequence $\varepsilon_n=\pm 1$ such that
\begin{equation*}
    \norm{\widehat{f}u}_q 
\gtrsim_q \Big\|{ \Big(\sum_{\substack{n\in I \\
   |n|\le M}} |\widehat{f_n}|^2\Big)^{\frac{1}{2}} u}\Big\|_q. 
\end{equation*}
Thus, since inequality \eqref{ineq:pittgen++} holds, by letting $M\to \infty$ we deduce that
\begin{eqnarray*}\nonumber C &\gtrsim_{q}&  \norm{ \left(\sum_{n \in I}|\widehat{f_n}|^2\right)^{\frac{1}{2}} u}_q 
\geq  \norm{ \left(\sum_{n \in I} \mathbbm{1}_{A(x_n/2,x_{n+1})} |\widehat{f_n}|^2\right)^{\frac{1}{2}} u}_q
\\
\label{zv}
&\approx_{q}& \left(\sum_{n \in I} \norm{ \mathbbm{1}_{A(x_n/2,x_{n+1})} \widehat{f_n}u}_q^q\right)^{\frac{1}{q}} 
\\
&\gtrsim_{q}& 
 \left(\sum_{n \in I} \left(\int_{  x_n }^{ x_{n+1}} u_0^{q^{\sharp}}(\xi) \xi^{d-1} d\xi \right)^{\frac{q}{q^{\sharp}}} 2^{\frac{qn}{2}}\right)^{\frac{1}{q}},
\end{eqnarray*} 
where we have used the fact that \begin{equation}     \label{vspom2}\sum_{n \in I} \mathbbm{1}_{A(x_{n}/2,x_{n+1})}(x) \leq 2\quad \forall x\in \mathbb{R}^d\end{equation}
and estimate \eqref{vspom1}. To see \eqref{vspom2}, observe that if $n''> n'>n$ and $n'\in I$, then we have $x_{n+1} \leq  x_{n'} < x_{n'+1}/2 \leq  x_{n''}/2   $. This means that  $ A(x_{n}/2, x_{n+1})\cap A(x_{n''}/2, x_{n''+1})= \emptyset$.

Thus, inequality \eqref{ineq:necstep2} is true.

\textbf{Step 2.} We now use inequalities \eqref{ineq:necstep1} and \eqref{ineq:necstep2} in order to prove \eqref{eq:step1}. 

First, recall that  $n\in J$ if $y_n \leq 2 y_{n+1}$. Hence,
\begin{equation*}\int_{y_n^{d}/2^d }^{ y_n^d } v_*^{-2}  \geq \int_{y_{n+1}^{d}}^{ y_n^d } v_*^{-2} \approx 2^n.\end{equation*} Thus, since $v_*$ is non-decreasing, for $ y_n^d \geq x^{-1} \geq  y^d_{n+1} \geq  y_n^d /2^d $ we have
\begin{eqnarray*} 
x\Big( \int_0 ^{x^{-1}} v_*^{-1} \Big)^2 
&\approx&  
y_n^{-d}\Big( \int_0 ^{y_n^d} v_*^{-1} \Big)^2 
\gtrsim
y_n^{d}v_*^{-2}(y_n^d/2^d)
\\&\geq&
\int_{ y_n^d/2^d}^{ y_n^d} v_*^{-2} 
\gtrsim 
2^n \approx \int_{y^d_{n+1}}^\infty v_*^{-2} \geq \int_{1/x}^\infty v_*^{-2}.
\end{eqnarray*}
In other words, there exists  $K>0$ such that
$$x\Big( \int_0 ^{x^{-1}} v_*^{-1} \Big)^2 
\geq K
 \int_{1/x}^\infty v_*^{-2},
\quad x_n^d \leq x \leq  x^d_{n+1}, \quad n\in J.
$$
We say that $x\in Z$ if the last inequality holds.
Note that, by the previous observation, \begin{equation}
\label{prop:JZ}
    \bigcup_{n \in J} [ x_n^d,x_{n+1}^d] \subset Z.
\end{equation}

Second, recall that $F\prec 1/v$, that is,
$$ \int_0 ^x F^{*,2} \lesssim \int_0 ^x \Big(\int_0 ^{1/t} v_*^{-1}\Big)^2 dt \approx x \Big(\int_0 ^{1/x} v_*^{-1} \Big)^2 + \int_{1/x}^\infty v_*^{-2},$$ and that for $x\in Z$ we have
$$  \int_0 ^x \Big(\int_0 ^{1/t} v_*^{-1}\Big)^2 dt\approx x \Big(\int_0 ^{1/x} v_*^{-1} \Big)^2.$$ 
Thus, for $x\in Z$,
\begin{equation*}x F ^{*,2} (x)\leq \int_0 ^x F^{*,2} \lesssim x\Big( \int_0 ^{1/x} v_*^{-1} \Big)^2,\end{equation*} that is,
\begin{equation*} F^* (x)  \lesssim \int_0 ^{1/x} v_*^{-1}, \quad x \in Z .\end{equation*} Hence, inequality \eqref{ineq:necstep1} implies that
$$\norm{F^* u^* \mathbbm{1}_Z}_q \lesssim _{q} C.$$ 
 What remains to show is that  \eqref{ineq:necstep2} and \eqref{prop:JZ} yield 
 \begin{equation}
 \label{eq:almost}
     \norm{F^* u^* \mathbbm{1}_{Z^c}}_q \lesssim _{q} C.
 \end{equation} \newline
 For any $x>0$, define $x':= \sup \left([0,x] \cap Z^c\right)\leq x $. Observe that, by definition of $Z$,
\begin{equation*} 
\int_0 ^x \mathbbm{1}_{Z^c} F^{*,2} \leq \int_0 ^{ x'}  F^{*,2} \lesssim\int_0 ^{ x'} \left(\int_0 ^{1/t} v_*^{-1} \right)^2 dt\lesssim \int_{1/x'} ^\infty v_*^{-2}  \lesssim \int_{1/x}^\infty v_*^{-2}.\end{equation*}  In particular, for $x=x^d_{n+1}$,

\begin{equation}
\label{ineq:obs}\int_{x^d_n}^{ x^d_{n+1}} \mathbbm{1}_{Z^c} |F ^* |^2 \lesssim \int_{y_{n+1}^d}^\infty v_*^{-2} \approx 2^n.
%, \quad n\in I.
\end{equation}
Therefore, by  \eqref{prop:JZ}, Hölder's inequality and inequality \eqref{ineq:obs}, we deduce that
 \begin{eqnarray*}  2^{-q}\norm{F^* u^* \mathbbm{1}_{Z^c}}_q^q&=&
\Big(\sum_{n\in I} +
\sum_{n\in J} \Big)
\int_{x_n^d}^{x_{n+1}^d} \mathbbm{1}_{Z^c} |u^*F ^* |^q
=
 \sum_{n\in I} \int_{x_n^d}^{x_{n+1}^d} \mathbbm{1}_{Z^c} |u^*F ^* |^q \\&\leq& \sum_{n\in I}  \left(\int_{ x_n^d}^{ x_{n+1}^d}  u^{*,q^{\sharp}} \right)^{\frac{q}{q^{\sharp}}} \left(\int_{ x_n^d}^{x_{n+1}^d} \mathbbm{1}_{Z^c} |F ^* |^2 \right) ^{\frac{q}{2}} 
\\&\lesssim_q& \sum_{n\in I} \left(\int_{ x_n^d}^{x_{n+1}^d}  u^{*,q^{\sharp}} \right)^{\frac{q}{q^{\sharp}}}  2^{\frac{qn}{2}} \lesssim_{q} C^q,
\end{eqnarray*}
where the last inequality follows from \eqref{ineq:necstep2}. Inequality \eqref{eq:almost} is thus proved, completing the proof.
\end{proof}

\begin{proof}[Proof of $(i) \implies (ii)$ in Theorem \ref{theorem:mainabstract}
for the general case] $ $\newline
    Assume that the inequality
    $ \norm{|\widehat{f}|^\beta}_X^{\frac1\beta} \leq C_1 \norm{ f}_Y$
    %    \eqref{ineq:pittXY}
     holds for any $f\in L^1$.

Let  $F\prec G$. 
Observe that, by right-admissibility of $Y$, for any $f\in L^1$, 
    $$\norm{|\widehat{f}|^{\beta}}_{X''}  ^{\frac1\beta} \leq \norm{|\widehat{f}|^{\beta}}_{X} ^{\frac1\beta} \leq C_1 \norm{f }_Y \leq 
    C_1\norm{f/G^\circ}_\infty \norm{G^\circ}_Y .$$
  Thus, for any $h \in X'$ with $\norm{h^\circ}_{X'}\leq 1$ we have,
  for any $f\in L ^1$,
  $$\int_{\mathbb{R}^d} h ^\circ |\widehat{f}|^{\beta} \leq \norm{h^\circ}_{X'} \norm{|\widehat{f}|^{\beta}}_{X''} \leq
  C_1^\beta\norm{f/G^ \circ}_\infty ^{\beta} \norm{G^\circ}_Y ^{\beta},
  $$
  that is, the inequality 
  \begin{equation*}
      \left(\int_{\mathbb{R}^d} h ^\circ |\widehat{f}|^{\beta} \right)^{\frac1\beta} \leq  (C_1 \norm{G^\circ}_Y)  \norm{f/G^ \circ}_\infty 
  \end{equation*} holds for any $f\in L^1$. Hence, applying Theorem \ref{theorem:mainabstract} 
in the particular case when $\||f|^\beta\|_X^\frac1\beta=\|fu\|_{q}$ with $q=\beta \leq 1,u^q=h^\circ$ and $
  \|f\|_Y=\|fv\|_{\infty}
  %Y=L^\infty(v)
  $ with $v=1/G^\circ$ and Corollary \ref{coro:method}, 
  we deduce that 
  \begin{equation*}
      \left(\int_{\mathbb{R}^d} h ^\circ F^{\circ,\beta} \right)^{\frac1\beta} \lesssim_\beta
      (C_1 \norm{G^\circ}_Y)
%      C_1 \norm{G^\circ}_Y 
\norm{G^\circ /G^ \circ}_\infty =  C_1 \norm{G^\circ}_Y.
  \end{equation*} The result follows by noting that, by left-admissibility of $X$ and the Hardy-Littlewood rearrangement inequality, we have
$$ \norm{F^{\circ,\beta}}_{X} = \norm{F^{\circ,\beta}}_{X''} = \sup _{\norm{h}_{X'}\leq 1} \int_{\mathbb{R}^d} |h F^{\circ,\beta}| \leq \sup _{\norm{h^\circ }_{X'}\leq 1} \int_{\mathbb{R}^d} h ^\circ F^{\circ,\beta}  \lesssim_\beta C_1^\beta \norm{G^\circ}_Y^{\beta}.$$
Corollary \ref{coro:method} completes the proof. 

\end{proof}

\section{Applications: Fourier inequalities in weighted Lebesgue spaces}
\label{section:finalremarks}
\subsection{
Solution to Problem B and discussion}

The main result of this subsection is the solution to Problem B for the Fourier transform on $\mathbb{R}^d$. Theorem \ref{theorem:mainh} characterizes, for the whole range of parameters $1\leq p \leq \infty$ and $0<q\leq \infty$,
% that is,  we find %complete
%      a complete characterization of 
 the radial non-increasing weights
 $u$ and $v^{-1}$ for which \eqref{ineq:pittgen} holds. The analogous results for Fourier series and coefficients 
 are discussed in Subsection 
 \ref{section5.2}.
% can be deduced from Theorem \ref{theorem:mainh} by means of Proposition \ref{rem2}.
%     
\begin{comment}
         This concludes the work started by several authors in 80s who obtain this  characterization for the case $p\le q$. The case $p> q$ was open up to now.
    \end{comment}
 
 %we characterize the weights $u$ and $v$ for which the inequality 
% $\norm{(F)^{\circ,\beta}}_X^{\frac1\beta}\leq C_3 \norm{G^\circ}_Y
% $
%  %\eqref{ineq:FGxy}
%   holds for any $F\prec G$, thus establishing the proof 
% of Theorem \ref{theorem:mainh}. The proof of Corollary \ref{coro:optiHL} follows by a small modification of the proof 
% of Theorem \ref{theorem:mainh}.
\begin{theorem}%[Fourier transform]
\label{theorem:mainh}
    Let $1\leq p \leq \infty$ and $0<q\leq \infty$. Assume that $u$ and $v^{-1}$ are radial non-increasing weights. Define     \begin{equation*}
     U(t)=\int_0 ^t u^{*,q},\quad V(t)= \int_0 ^t v_*^{p},\quad\mbox{and}\quad
     \xi(t) = U(t)+ t^{\frac{q}{2}} \left(\int_t ^\infty u^{*,q^{\sharp}} \right)^{\frac{q}{q^{\sharp}}}.\end{equation*}
     Then
     inequality \eqref{ineq:pittgen} holds if and only if
    \begin{enumerate}[label=(\roman*)]
    \item $p \leq q$ and  \begin{equation}\label{condition:11}
        %K\approx
        C_3:=\sup _{s>0} \left(\int_0^s u^{*,q} \right)^{\frac{1}{q}} \left(\int_0^{1/s} v_*^{-p' } \right)^{\frac{1}{p'}} <\infty; 
        \end{equation}
        \item $1\leq q<p$, $2\leq\max(q,p'),$ and  \begin{equation}\label{nec11}
      % K\approx
      C_4:= \left(
       \int_0^\infty u^{*,q}(s)
       \left(\int_{0}^{s} u^{*,q}\right)^{\frac{r}{p}}\left(\int_{0}^{1/s} v_*^{-p'}\right)^{\frac{r}{p'}}  ds \right)^{\frac{1}{r}}<\infty;
    \end{equation}
%   where $r^{-1}=q^{-1}-p^{-1};$

     \item $q<2< p<\infty$, %, $\max(q,p')< 2$, 
     \begin{equation}
         \int_1 ^\infty u^{*,q^{\sharp}} < \infty
     \end{equation} and  
     \begin{equation}
         \label{eq:monster}
C_5=C_4+C_6<\infty,
     \end{equation}
where %$\mathfrak{p}$
     \begin{eqnarray*}
C_6&=&\Bigg(\int_0^\infty u^{*,q}(t) U^{\frac{q^{\sharp}}{2}}(t) \xi^{-\frac{q^{\sharp}}{2}}(t) t^{-\frac{q}{2}} 
\\&\times&\left(\int_{t}^ \infty u^{*,q}(r) U^{\frac{q^{\sharp}}{2}}(r) \xi^{-\frac{q^{\sharp}}{2}}(r) r^{-\frac{q}{2}}   dr \right)^{\frac{r}{p}} \left( \int_{t^{-1}} ^\infty v_*^{-p^{\sharp}}(r) \left(\frac{V(r)}{r v_*^{p}(r)}\right)^{-\frac{p^{\sharp}}{2}} dr \right) ^{\frac{r}{p^{\sharp}}}dt\Bigg)^{\frac{1}{r}};
     \end{eqnarray*}
 \item $q<2< p=\infty$, %, $\max(q,p')< 2$, 
     \begin{equation}
         \int_1 ^\infty u^{*,q^{\sharp}} < \infty
     \end{equation} and  
     \begin{equation}
   \label{eq:pinfty}
   C_7:= \left(\int_0 ^\infty U^{\frac{q^{\sharp}}{2}}(t) u^{*,q}(t) \xi^{-\frac{q^{\sharp}}{2}}(t) t^{-\frac{q}{2}} \left(\int_0 ^t \left ( \int_0 ^{r^{-1}} v_*^{-1} \right)^2 dr \right)^{\frac{q}{2}}dt\right)^{\frac{1}{q}}<\infty;
\end{equation}

\item $q<1\leq p\leq 2$,

     \begin{equation}
         \int_1 ^\infty u^{*,q^{\sharp}} < \infty,
     \end{equation} and  
     \begin{equation}
         \label{eq:monster2}
C_8=C_4+C_{9}<\infty,
     \end{equation}

\end{enumerate}
    where
     \begin{eqnarray*}
C_{9}&=&\Bigg(\int_0^\infty u^{*,q}(t) U^{\frac{q^{\sharp}}{2}}(t) \xi^{-\frac{q^{\sharp}}{2}}(t) t^{-\frac{q}{2}} 
\\&\,&\qquad\quad\times\left(\int_{t}^ \infty u^{*,q}(r) U^{\frac{q^{\sharp}}{2}}(r) \xi^{-\frac{q^{\sharp}}{2}}(r) r^{-\frac{q}{2}}   dr \right)^{\frac{r}{p}} \Big(\sup_{1/t\leq y<\infty} \frac{y ^{\frac{r}{2}}}{V^{\frac{r}{p}}(y)}
\Big)dt\Bigg)^{\frac{1}{r}}.
     \end{eqnarray*}     
Moreover, the optimal constant $C$ in the % Pitt's
inequality $
 \big\|u\widehat{f}\,\big\|_{q}\leq C \norm{ f v}_p
 $
 satisfies
$$C\approx_{p,q} C_j, \qquad j=3,4,5,7,8,
$$
in items (i)-(v), respectively.
\end{theorem}

\begin{remark}
 \begin{enumerate}[label=(\roman*)]
        \item  For general weights $u$ and $v$
  and for general operators $T$ of joint strong type
  $(1,\infty;2,2)$,
  the conditions in Theorem \ref{theorem:mainh} are sufficient for $\|uT(f)\|_q\le C \|vf\|_p$, in other words, $C\lesssim_{p,q} C_j,$ $j=3,4,5,7,8,
$ in the cases of items (i)-(v), respectively.

    \item In the case when $q=\infty$ or $p=1$, $C=
    \|u\|_q\|1/v\|_{p'}$, see Proposition \ref{theorem:degheinig}.
    
    \item   Note that by duality (see Lemma \ref{lemma:duality}), for $1\leq p,q \leq \infty$, the conditions in Theorem \ref{theorem:mainh} are equivalent to those obtained by exchanging $(u,v,p,q)  $ and $(v^{-1},u^{-1},q',p')$.

\end{enumerate}
   
\end{remark}

Our proof of Theorem \ref{theorem:mainh} is based on  
the following corollary of Theorem \ref{theorem:mainabstract}. Here, again $\widehat{f}$ stands for the Fourier transform.

%we first prove Corollary \ref{coro:optiHL} by combining Theorem \ref{theorem:mainabstract} with the results of Subsection \ref{subsec:hardy}. The proof of Theorem \ref{theorem:mainh} then follows from Corollary \ref{coro:optiHL} and the results of Subsection \ref{subsec:hardy}.

\begin{corollary}
\label{coro:optiHL}
    Let $u$ be radial non-increasing and $0<q<\infty$. Define     \begin{equation*}
     U(t)=\int_0 ^t u^{*,q}\quad\mbox{and}\quad
     \xi(t) = U(t)+ t^{\frac{q}{2}} \left(\int_t ^\infty u^{*,q^{\sharp}} \right)^{\frac{q}{q^{\sharp}}}.\end{equation*} Then, the largest right-admissible space $Y$  for which either
%    \begin{enumerate}[label=(\roman*)]
 %       \item 
 \begin{equation}
 \label{eq:optihl}
        \norm{u\widehat{f}\,}_q\leq \norm{f}_Y
            \end{equation}
  or, equivalently, for any $T$ of joint strong type $(1,\infty;2,2)$,
 %      \item 
 \begin{equation*}
        \norm{u T(f)}_q\leq \norm{f}_Y \mbox{ hold}
            \end{equation*}  
 %   \end{enumerate}
satisfies
\begin{enumerate}[label=(\roman*)]
    \item if $0<q<2$ and $\xi(t)<\infty,$  \begin{equation}
 \label{eq:kgoptimal}
\norm{f}_Y \approx_q   \left(\int_0 ^\infty U^{\frac{q^{\sharp}}{2}}(t) u^{*,q}(t) \xi^{-\frac{q^{\sharp}}{2}}(t) t^{-\frac{q}{2}} \left(\int_0 ^t \left ( \int_0 ^{r^{-1}} f^* \right)^2 dr \right)^{\frac{q}{2}}dt\right)^{\frac{1}{q}};
 %\lesssim \norm{G^*v_*}_p,
 \end{equation}
 \item if $0<q<2$ and $\xi(t)=\infty,$  \begin{equation}
\norm{f}_Y =\infty;
 %\lesssim \norm{G^*v_*}_p,
 \end{equation}
 \item if $2\leq q<\infty$, or if $q \geq 1$  and $\xi(t)\approx U(t),$
 \begin{equation}
 \label{eq:kgoptimal2}
\norm{f}_Y \approx_q   \left(\int_0 ^\infty u^{*,q}(t) \left( \int_0 ^{t^{-1}} f^*   \right)^{q} dt\right)^{\frac{1}{q}}.
 \end{equation}
\end{enumerate}

\end{corollary}

\subsection{Counterparts for Fourier series and coefficients}\label{section5.2}
   We now explain how 
to modify
Corollary \ref{coro:optiHL} in order to address   the cases of Fourier series and Fourier coefficients.
\begin{remark}
(i) If inequality \eqref{eq:optihl} is replaced by \begin{equation}
 \left(\int_{-\frac{1}{2}}^{\frac{1}{2}}  \left| u(x)\sum_{n=-\infty}^\infty f(n)e^{2 \pi i n x}\right|^q dx \right)^\frac{1}{p} \leq \norm{f}_Y,
\end{equation} where $f$ is a summable sequence, %from the proof of Corollary \ref{coro:optiHL} we see that 
Corollary \ref{coro:optiHL} holds true with
$u=u\mathbbm{1}_{[-\frac12,\frac12]}$
and
$\displaystyle\int_0^{r^{-1}} f^*:= \int_0^{r^{-1}}\sum_{n=0}^\infty f^{*}(n) \mathbbm{1}_{[n,n+1]}$.
\\(ii) If inequality \eqref{eq:optihl} is replaced by  \begin{equation}
    \left(\sum_{n=-\infty}^\infty  | u_n\widehat{f}(n)|^q\right)^\frac{1}{q} \leq \norm{f}_Y,
\end{equation} where $f$ is an integrable function defined on $[-\frac 12, \frac 12]$ and $\widehat{f}(n)$ are its Fourier coefficients, Corollary \ref{coro:optiHL} holds true with $u(t):=\widetilde{u}(t)$, where
for a sequence $w$ we define $\displaystyle\widetilde{w}(x) =\sum_{n=-\infty}^{\infty} w_n \mathbbm{1}_{[n-\frac{1}{2}, n+\frac{1}{2}]}(x)$.
\end{remark}

Now we discuss analogues of Theorem \ref{theorem:mainh}  in the  discrete setting, i.e.,
 %Third, to obtain characterizations 
  characterizations of the even monotone weights $u$ and $v$ for which the inequalities
\begin{equation}
\label{eq:disc2}
 \left(\int_{-\frac{1}{2}}^{\frac{1}{2}}  \left| u(x)\sum_{n=-\infty}^\infty f(n)e^{2 \pi i n x}\right|^q dx \right)^\frac{1}{p} \leq C_{10} \left(\sum_{n=-\infty}^\infty  |v_n f(n)|^p\right)^\frac{1}{p} 
\end{equation}
and
\begin{equation}
\label{eq:disc1}
    \left(\sum_{n=-\infty}^\infty  | u_n\widehat{f}(n)|^q\right)^\frac{1}{q} \leq C_{11} \left(\int_{-\frac{1}{2}}^{\frac{1}{2}} |vf|^p \right)^\frac{1}{p}
\end{equation}
hold.
These estimates have been studied in  
\cite{oscar},
\cite{h}, \cite{jurkat}, \cite{rastegari}, \cite{stein}, and \cite[Chapter XII]{z}, 
as well as  in many other papers.
We claim that 
 % it is useful to 
 %We prove in Proposition \ref{rem2} that 
  these inequalities are equivalent to the corresponding “continuous” estimates, 
 namely,
%without having to rederive the results of Theorem \ref{theorem:mainh} in the discrete case, it is useful to note that they are equivalent to the "continuous" inequalities, whose characterizations can be read from Theorem \ref{theorem:mainh},

\begin{equation}
\label{eq:cont2}
 \left(\int_{-\infty}^{\infty} \mathbbm{1}_{[-\frac{1}{2},\frac{1}{2}]}|u\widehat{f}|^q \right)^\frac{1}{p} \leq C_{12}  \left(\int_{-\infty}^\infty |\widetilde{v} f| ^p\right)^\frac{1}{p} 
\end{equation}
and \begin{equation}
\label{eq:cont1}
    \left(\int_{-\infty}^\infty  | \widetilde{u}\widehat{f}|^q\right)^\frac{1}{q} \leq C_{13} \left(\int_{-\infty}^{\infty} \mathbbm{1}^{-1}_{[-\frac{1}{2},+\frac{1}{2}]}|vf|^p \right)^\frac{1}{p},
\end{equation} respectively.
 %respectively.
 The  characterizations of \eqref{eq:cont2} and 
 \eqref{eq:cont1}
 follow from Theorem 
 \ref{theorem:mainh}.

\begin{proposition}
    \label{rem2} We have $C_{10}\approx_{p,q} C_{12}$
and $C_{11}\approx_{p,q} C_{13}$.
\end{proposition}

\begin{proof} We only show the equivalence $\eqref{eq:disc2} \iff \eqref{eq:cont2}$, the proof of $\eqref{eq:disc1} \iff \eqref{eq:cont1}$ is similar.

    The inequality  $C_{10}\lesssim_{p,q} C_{12}$ follows easily by defining $$f=\sum_{n\in \mathbb{Z}} f(n) \mathbbm{1}_{[n-\frac12,n+\frac12]}$$ and observing that $$\widehat{f}(\xi)= \frac{\sin{\pi \xi}}{\pi \xi}\sum_{n\in \mathbb{Z}} f(n)e^{-2\pi i \xi n}.$$ 
    
    We now obtain the reverse estimate. Assume  that $\eqref{eq:disc2}$ holds, then by Theorem \ref{theorem:mainabstract} and Corollary \ref{coro:method} we deduce that for any $G\prec B$ the inequality
    \begin{equation*}
         \left(\int_{-\frac{1}{2}}^{\frac{1}{2}}  \left| u G\right|^q \right)^\frac{1}{q} \lesssim_{p,q} C_{10}  \left(\sum_{n=-\infty}^\infty  |v_n B(n)|^p\right)^\frac{1}{p} 
    \end{equation*} holds. Thus, to prove inequality \eqref{eq:cont2} it suffices to show that
    \begin{equation}
    \label{eq:GB}
        \widehat{f}\, \mathbbm{1}_{[-\frac12,\frac 12]} \prec F,
    \end{equation}
    where $F=\sum_{n \in \mathbb{Z}} \mathbbm{1}_{[-\frac 12, \frac 12]}\left( \int_{n-\frac 12}^{n+\frac 12} |f| \right).$

    To see this, observe that by Lemma \ref{lemma:kfunc} applied to the function $g=f*\mathbbm{1}_{[-\frac 12,\frac 12]}$ we have that
for any $x>0$,
\begin{equation*}
    \int_0^ x(\widehat{g})^{*,2} \lesssim \int_\frac 1x^\infty  g^{**,2}.
\end{equation*}
    Inequality \eqref{eq:GB} will be proved if we show that $$ |\widehat{f}(\xi)| \mathbbm{1}_{[-\frac12,\frac 12]}(\xi)\lesssim |\widehat{g}(\xi)| \qquad \mbox{and} \qquad g^{**}\lesssim F^{**}.$$ The former estimate is a simple consequence of the formula $\widehat{g}(\xi)=\widehat{f}(\xi) \frac{\sin{\pi \xi}}{\pi \xi}$. To prove the latter inequality, we observe that $g(x)\leq F(x-\frac 12)+ F(x) +F(x+\frac 12)$, which implies the result.
    \end{proof}

\subsection{Proofs of Theorem \ref{theorem:mainh}
and Corollary \ref{coro:optiHL}
}

\begin{proof}[Proof of Corollary \ref{coro:optiHL}]
Taking into account Corollary \ref{coro:mainbis}  with $\norm{|F|^\beta}_X^\frac1\beta=\norm{F^*u^*}_q$, 
the optimal $Y$ is given by 
$\widetilde{X}$, where, by  Corollary \ref{coro:method}, 
 \begin{equation*}\norm{G}_{\widetilde{X}}= \sup_{F \prec G} \norm{F^*u^*}_q. \end{equation*} 
We now compute this supremum for different values of $q$.

 {\bf Item (iii) for $q\geq 2$.} %\newline
In this case  the technique to obtain the result is essentially known (see the survey \cite{BHJFAA}). Indeed, the trivial observation $\int_0 ^{t^{-1}}G^* \prec G^*$ yields
    $$\norm{G}_{\widetilde{X}} \gtrsim   \left(\int_0 ^\infty u^{*,q}(t) \left( \int_0 ^{t^{-1}} G^*   \right)^{q} dt\right)^{\frac{1}{q}}.$$ For the reverse inequality, note that a convexity argument \cite[Theorem 4.7]{jodeit} shows that if $F\prec G$, then for $q\geq 2$
    $$\int_0 ^x F^{*,q}  \lesssim_q \int_0 ^x \left(\int_0 ^{t^{-1}} G^* \right)^q dt.$$ Thus, for any $F\prec G$, integration by parts yields 
    $$\norm{F^* u^*}_q \lesssim _q  \left(\int_0 ^\infty u^{*,q}(t) \left( \int_0 ^{t^{-1}} G^*   \right)^{q} dt\right)^{\frac{1}{q}},$$ whence the result follows.

{\bf Items (i) and (ii).}
In order to estimate $\norm{G}_{\widetilde{X}}$ in the case $0<q<2$, we observe that
since $$F \prec \Big(\sup _{x>0} \frac{\int_0 ^x F^{*,2}}{\int_0 ^x \left(\int_0 ^{t^{-1}} G^*\right)^2 dt}\Big)^\frac12 G ,$$ 
$\norm{G}_{\widetilde{X}}$ is the smallest constant $K$ for which the inequality
 \begin{equation}
 \label{ineq:dev}
     \left(\int_0 ^\infty u^{*,q} F^{*,\frac{q}{2}} \right)^{\frac{2}{q}} \leq K %\norm{G}_{\widetilde{X}}^2 
     \sup _{x>0} \frac{\int_0 ^x F^{*}}{\int_0 ^x \left(\int_0 ^{t^{-1}} G^*\right)^2 dt}
 \end{equation}
 holds for any $F^*$.

 Applying Theorem \ref{theorem:krepela} with $$\mathfrak{q}=\frac{q}{2},\quad w=u^{*,q},\quad \nu(x) = \frac{x}{\int_0 ^x \left(\int_0 ^{t^{-1}} G^*\right)^2 dt}$$
 and noting that $\nu(x)$ is non-decreasing and $\nu(x)/x$ is non-increasing, we deduce that
 in order for inequality \eqref{ineq:dev} to hold with a finite constant it is necessary that $
 \xi(t)<\infty$ for any positive $t.$ 
 Furthermore,  in this case,
 \begin{equation}
 \label{eq:kg}
 \norm{G}_{\widetilde{X}} \approx_q   \left(\int_0 ^\infty U^{\frac{q^{\sharp}}{2}}(t) u^{*,q}(t) \xi^{-\frac{q^{\sharp}}{2}}(t) t^{-\frac{q}{2}} \left(\int_0 ^t \left ( \int_0 ^{r^{-1}} G^* \right)^2 dr \right)^{\frac{q}{2}}dt\right)^{\frac{1}{q}}
 %\lesssim \norm{G^*v_*}_p,
 \end{equation}
 with  $U(t)=\int_0 ^t u^{*,q}$
 and
\begin{equation*}\xi(t)=t^{\frac{q}{2}} \left(\int_t^\infty \Big(\frac{U(s)} {s}\Big)^{\frac{q^{\sharp}}{q}} ds\right)^{\frac{q}{q^{\sharp}}} \approx_q U(t) + t^{\frac{q}{2}} \Big(\int_t ^\infty u^{*,q^{\sharp}}\Big)^{\frac{q}{q^{\sharp}}},\end{equation*} where the last estimate follows from the Hardy's inequality given by  Lemma \ref{theorem:cont hardy2} and monotonicity. 
The proof of items (i) and (ii) is now complete.

{\bf Item (iii) for $1\leq q< 2$.}
We now show that if $\xi\approx U$, equation \eqref{eq:kg} can be simplified to 
  $$\norm{G}_{\widetilde{X}} \approx_q \left(\int_0 ^\infty u^{*,q}(t) \left( \int_0 ^{t^{-1}} G^*   \right)^{q} dt\right)^{\frac{1}{q}}.$$ 
As shown in \cite[p. 1233]{sin}, it suffices to prove    \begin{equation}
    \label{eq:Bp}
    y^{\frac{q}{2}} \int_y ^\infty u^{*,q} (x) x^{-\frac{q}{2}} dx \lesssim_q \int_0 ^\infty u^{*,q} \min(1, y/x)^q dx,\quad y>0,
\end{equation} see \cite[Theorem 2.1, condition (c)]{sin}.

 Observe that by a scaling argument it is enough to establish \eqref{eq:Bp} for $y=1$, 
 which, 
  by the hypothesis $\xi\approx U$,  is implied by  the inequality
  \begin{equation}
\label{ineq:xiU}
\int_1^\infty u^{*,q}(t) t^{-\frac{q}{2}} dt \lesssim_{q} \int_0^1 u^{*,q} + \left(\int_1^\infty u^{*,q^\sharp}\right)^\frac{q}{q^\sharp}.
\end{equation}
First, we claim that for any integer $k\geq 0$ 
\begin{equation}
\label{ineq:xiU2}
\left(\int_1^\infty u^{*,q^\sharp} (t)\log^{k}(t+1) dt\right)^\frac{q}{q^\sharp} \lesssim_{q,k} \int_0^1 u^{*,q}+ \left(\int_1^\infty u^{*,q^\sharp}  \right)^\frac{q}{q^\sharp}.
\end{equation} Indeed,  observe that since $\xi\approx U$ and by monotonicity
\begin{align*}
\int_1  ^\infty u^{*,q^\sharp} (t)\log (t+1)^{k+1}dt  &\approx_{k} \int_1  ^\infty t^{-1} \log(t+1)^{k} \left(\int_t ^\infty u^{*,q^\sharp}\right)  dt\\
    &\lesssim_q \int_1 ^\infty \frac{\log(t+1)^k}{t^{\frac{q^\sharp}{q}} } \left(\int_0 ^t u^{*,q} \right)^{\frac{q^\sharp}{q}}dt\\
 &\lesssim_{q,k} \left(\int_0 ^1 u^{*,q}\right)^{\frac{q^\sharp}{q}} + \int_1 ^\infty u^{*,q^\sharp} (t)\log(t+1)^kdt,
 \end{align*} where the last estimate follows from Hardy's inequality (Lemma \ref{theorem:cont hardy2}). Inequality \eqref{ineq:xiU2} follows by iterating the previous argument. Finally, by Hölder's inequality and \eqref{ineq:xiU2}, for a large enough $k$ we conclude that
 \begin{equation*}
     \int_1^\infty u^{*,q}(t) t^{-\frac{q}{2}} dt \lesssim_{q,k} \left(\int_1^\infty u^{*,q^\sharp} (t)\log^{k}(t+1) dt\right)^\frac{q}{q^\sharp} \lesssim_{q,k} \int_0^1 u^{*,q}+ \left(\int_1^\infty u^{*,q^\sharp}  \right)^\frac{q}{q^\sharp},
 \end{equation*} that is, \eqref{ineq:xiU} holds.
  
\end{proof}

\begin{proof}[Proof of Theorem \ref{theorem:mainh}]
Observe that, by Corollary \ref{coro:mainbis}, inequality \eqref{ineq:pittgen} holds if and only if for any $G$ \begin{equation}
 \label{constant K}\norm{G}_{\widetilde{X}} \lesssim \norm{G^*v_*}_p,\end{equation} where $\widetilde{X}$ is given in Corollary \ref{coro:optiHL}.
We study inequality \eqref{constant K} separately for each regime of $p$ and $q$. 
 
 {\bf Items (i) and (ii).} Item (i) and the estimate $C\lesssim_{p,q}C_4$ in item (ii) are well known, see \cite{BHJFAA}. The estimate $C\gtrsim_{p,q}C_4$  follows from Theorem \ref{lemma:necesnew}.

{\bf Item (iv).}
In light of Corollary \ref{coro:optiHL} (i), we see that inequality \eqref{constant K} holds for any $G$ if and only if $\xi(t)<\infty$ and 
$\|v^{-1}_*\|_{\widetilde{X}}\lesssim 1$, which are equivalent to both $\int_1 ^\infty u^{*,q^{\sharp}} < \infty$
  and
  $C_7\lesssim 1$.

{\bf Preliminaries for items (iii) and (v).} %Step 3.  The case $p<\infty$.}
First, by Hardy's inequality and monotonicity, we note that   
 \begin{equation*}\left(\int_0 ^t \Big ( \int_0 ^{r^{-1}} G^* \Big)^2 dr \right)^{\frac{q}{2}} \approx_q t^{\frac{q}{2}} \left(\int_0 ^{t^{-1}} G^*\right)^q + \left(G^{*,2}(t^{-1})t^{-1}+\int_{t^{-1}}^\infty G^{*,2}\right)^{\frac{q}{2}}.\end{equation*}
 Thus, in view of  Corollary \ref{coro:optiHL} (i), inequality
 \eqref{constant K} 
  is equivalent to the following three conditions:  \begin{equation}\label{xi}\xi(t)<\infty,\end{equation}
 \begin{equation}
 \label{eq:Kg1}
     \left(\int_0 ^\infty U^{\frac{q^{\sharp}}{2}}(t) u^{*,q}(t) \xi^{-\frac{q^{\sharp}}{2}}(t)  \Big(\int_0 ^{t^{-1}} G^*\Big)^q\right)^{\frac{1}{q}} \lesssim_q \norm{G^*v_*}_p
 \end{equation} 
 and, writing $G^*$ instead of $G^{*,2}$,\begin{equation}
 \label{eq:Kg2}
     \left(\int_0 ^\infty U^{\frac{q^{\sharp}}{2}}(t) u^{*,q}(t) \xi^{-\frac{q^{\sharp}}{2}}(t) t^{-\frac{q}{2}} \Big(G^{*}(t^{-1})t^{-1}+\int_{t^{-1}}^\infty G^{*}\Big)^{\frac{q}{2}}\right)^{\frac{2}{q} \cdot \frac12} \lesssim_q \norm{G^*v_*^2}^\frac12_{\frac p2}.
 \end{equation}

Second, we deal with inequality \eqref{eq:Kg1}. Since $\xi\gtrsim U,$  we see that \eqref{eq:Kg1} is implied by the simpler estimate\begin{equation}
 \label{eq:Kg11}
    \left(\int_0 ^\infty  u^{*,q}(t)  \Big(\int_0 ^{t^{-1}} G^*\Big)^q\right)^{\frac{1}{q}} \lesssim \norm{G^*v_*}_p,
 \end{equation} which, by Lemma \ref{theorem:cont hardy2}, is characterized by condition $C_4\lesssim 1$.
 
 Moreover, noting that trivially $\int_0 ^{t^{-1}}G^* \prec G^*$ 
 (or Corollary \ref{coro:FH}), we see that \eqref{eq:Kg11} is also necessary for inequality \eqref{constant K}. In other words,
  \eqref{constant K} is equivalent to the conditions
 \eqref{xi},
 \eqref{eq:Kg2}, and
 \eqref{eq:Kg11}.
 
Since \eqref{xi}  and \eqref{eq:Kg11} are equivalent to the conditions
  $\int_1 ^\infty u^{*,q^{\sharp}} < \infty$
  and
  $C_4\lesssim 1$, all that remains is to characterize inequality \eqref{eq:Kg2} for $p<\infty$.
 
% Hence, nothing is lost in this replacement. 
{\bf Item (iii).} %Step 3.1.  The case $0<q<2 < p<\infty$.}
%Second, we characterize inequality \eqref{eq:Kg2}. First we consider the case $p\geq 2$.
By Theorem \ref{theorem:gogatoriginal} with $b=v_*^p$, $\mathfrak{p}=\frac{p}{2}$, and noting that, by monotonicity, $\int_0 ^\infty v_*^p = \infty$, we see that, since

$$t^{-1} \int_{t^{-1}}^\infty  g + \int_{t^{-1}}^\infty \left(\int_s^\infty g \right) ds =\int_{t^{-1}}^\infty xg(x) dx,$$
inequality \eqref{eq:Kg2} holds for any non-increasing $G^*$ if and only if for any non-negative $g$ 
 \begin{equation*}
 \label{eq:Kg3}
     \left(\int_0 ^\infty U^{\frac{q^{\sharp}}{2}}(t) u^{*,q}(t) \xi^{-\frac{q^{\sharp}}{2}}(t) t^{-\frac{q}{2}} \Big(\int_{t^{-1}} ^\infty x g (x) dx\Big)^{\frac{q}{2}}dt\right)^{\frac{1}{q}} \lesssim \norm{g  v_*^{2-p} \int_0 ^x v_*^p}_{\frac p2}^\frac12,
 \end{equation*} or, equivalently,
  \begin{equation*}
      \left(\int_0 ^\infty U^{\frac{q^{\sharp}}{2}} (t^{-1})u^{*,q}(t^{-1}) \xi^{-\frac{q^{\sharp}}{2}}(t^{-1}) t^{\frac{q}{2}-2} \Big(\int_{t} ^\infty  g \Big)^{\frac{q}{2}}dt\right)^{\frac{1}{q}} \lesssim \norm{g   v_*^{2-p} \Big(\frac1x\int_0 ^x v_*^p\Big)}_{\frac p2}^\frac12.
 \end{equation*}
 By Lemma \ref{theorem:cont hardy} with $\mathfrak{q}=\frac q2$ and $\mathfrak{p}=\frac p2>1$, the sharp constant in the previous inequality is equivalent to  $C_6$. This completes the proof of this case.

{\bf Item (v).}
%Step 3.2. The case $0<q<1\leq p\leq2$.}
% Finally, we consider the case $0<q<1 \leq p<2$.
By Theorem \ref{th:oinarov}, inequality \eqref{eq:Kg2} is equivalent to both 
 \begin{multline}
 \label{eq:Kgaug}
     \left(\int_0 ^\infty U^{\frac{q^{\sharp}}{2}} (t^{-1})u^{*,q}(t^{-1}) \xi^{-\frac{q^{\sharp}}{2}}(t^{-1}) t^{\frac{q}{2}-2} \left(\int_{t}^{\infty} g(y) (y-t)^{\frac{p}{2}}dy\right)^{\frac{q}{p}}dt\right)^{\frac{p}{q}\cdot\frac{1}{p}}\\ \lesssim \Big(\int_0 ^\infty g(x) \Big(\int_0 ^x v_* ^p\Big) dx\Big)^\frac1p
 \end{multline} for any non-negative $g$ and 
  \begin{equation}
 \label{eq:Kgaug2}
     \left(\int_0 ^\infty U^{\frac{q^{\sharp}}{2}}(t) u^{*,q}(t) \xi^{-\frac{q^{\sharp}}{2}}(t) t^{-\frac{q}{2}} \left(G^{*}(t^{-1})t^{-1}\right)^{\frac{q}{2}}dt\right)^{\frac{2}{q}\cdot\frac{1}{2}} \lesssim \norm{G^*v_*^2}_{\frac p2}^
     \frac{1}{2}
 \end{equation} for any non-increasing $G^*$. By a change of variables and letting $G^{*}(t)=\Big(\int_t^\infty g\Big)^\frac2p$, inequality \eqref{eq:Kgaug2} can be equivalently written as 
   \begin{multline}
 \label{eq:Kgaug23}
     \left(\int_0 ^\infty U^{\frac{q^{\sharp}}{2}} (t^{-1})u^{*,q}(t^{-1}) \xi^{-\frac{q^{\sharp}}{2}}(t^{-1}) t^{\frac{q}{2}-2} \left(t^{\frac{p}{2}}\int_t ^\infty g(y) dy\right)^{\frac{q}{p}}dt\right)^{\frac{p}{q}\cdot\frac{1}{p}}\\ \lesssim \Big(\int_0 ^\infty g(x) \left(\int_0 ^x v_*^p\right) dx
     \Big)^\frac{1}{p}
\end{multline}
    for any non-negative $g$. By putting together inequalities \eqref{eq:Kgaug} and \eqref{eq:Kgaug23}, we conclude that inequality \eqref{eq:Kg2} is equivalent to
  \begin{multline*}
     \left(\int_0 ^\infty U^{\frac{q^{\sharp}}{2}} (t^{-1})u^{*,q}(t^{-1}) \xi^{-\frac{q^{\sharp}}{2}}(t^{-1}) t^{\frac{q}{2}-2} \left(\int_{t}^{\infty} g(y) y^{\frac{p}{2}}\right)^{\frac{q}{p}}dt\right)^{\frac{p}{q}\cdot\frac{1}{p}}\\ \lesssim \Big(\int_0 ^\infty g(x) \left(\int_0 ^x v_*^p\right) dx
     \Big)^\frac{1}{p}
  \end{multline*}for any non-negative $g$. 
 
 Finally, by Lemma \ref{theorem:cont hardy},  the constant in the previous estimate is equivalent to $C_{9}$.
 The proof of Theorem \ref{theorem:mainh} is now complete.
 \end{proof}

\subsection{Extreme cases}

% \subsection{Degenerate cases}
 Here we 
 compute  the exact value of the constant $C$
 in \eqref{ineq:pittnomono}
 in some extreme cases of the parameters $p,q$.
\begin{remark}
\label{theorem:degheinig} 
(i) 
Let $1 \leq p \leq \infty $ and $0<q \leq \infty$.
Inequality \eqref{ineq:pittnomono} holds with $C=\norm{u}_q\norm{v^{-1}}_{p'}$. Moreover, if either $q=\infty$ or $p=1$, this constant is sharp.
See, e.g. \cite{sin}.
   
   (ii) We note that in the case $p=1$ both descriptions of the constant $C$, the one given by    Theorem \ref{theorem:mainh} (i,v) and
   $C=\norm{u}_q\norm{v^{-1}}_{p'}$,
    are equivalent. In more detail, $C_3\approx_{q} \norm{u}_q\norm{v^{-1}}_{\infty}$
for $1\leq q$ and $C_8\approx_{q} \norm{u}_q\norm{v^{-1}}_{\infty}$ for $0<q<1$.
 Similarly, 
$C_3\approx_{p} \norm{u}_\infty\norm{v^{-1}}_{p'}$ for $q=\infty.$

\end{remark}
\begin{proof} (i)
To see that \eqref{ineq:pittnomono} holds with the aforementioned constant, observe that
\begin{equation*} \norm{\widehat{f}u}_q \leq \norm{\widehat{f}}_\infty \norm{u}_q \leq \norm{f}_1 \norm{u}_q \leq \norm{fv}_{p} \norm{u}_q\norm{v^{-1}}_{p'}.\end{equation*} 
   For the sharpness of the constant, if $q= \infty$ and \eqref{ineq:pittnomono} holds, then for any $\xi\in \mathbb{R}^d$ 
   \begin{equation*}
       u(\xi) \left| \int_{\mathbb{R}^d} f(x) e^{-2 \pi i \langle \xi ,x \rangle} dx\right| \leq C \norm{fv}_p
   \end{equation*} 
 and the result follows by taking $f(x)=v^{-p'}(x)e^{2 \pi i \langle \xi ,x \rangle}$. For $p=1$, if \eqref{ineq:pittnomono} holds,
 we let $f(x)=\frac{\mathbbm{1}_{|x|\leq 1}}{\mu (\{|x|\leq 1\})}$ and take $f_{a,b}(x)=b^{-d}f\left(\frac{x-a}{b}\right)$.
Then  $|\widehat{f_{a,b}}(x)|=|
\widehat{f_{0,1}}(bx)|$, $
|\widehat{f_{a,b}}(0)|=1,
$ and the results  follows by letting $b\to 0.$ Note that in the case $0<p<1$ the same argument shows that
inequality \eqref{ineq:pittnomono} does not hold unless
$u=0$ or $v=\infty$. 

(ii) The proof is straightforward.
\end{proof}

\section{Applications: Optimal Fourier inequalities}
\label{section:optimalineq}
In this section, we present two particularly meaningful applications of Corollary \ref{coro:optiHL}. 
First, we study Fourier estimates for Morrey-type spaces and, in particular, 
establish the optimal Fourier inequality \eqref{eq:expL} for the exponential Orlicz space. Second, we answer the question raised by Herz and Bochkarev regarding  the optimal Fourier inequality for $L^{2,p}(\mathbb{T})$ for $p>2$; see \eqref{eq:sharpboch}.

%In Subsection \ref{subsec:morrey} we establish, among other results, the optimal Fourier inequality \eqref{eq:expL} for the exponential Orlicz space. The main result of Subsection \ref{subsection:FS} is the obtention of the optimal inequality \eqref{eq:sharpboch} for the Fourier coefficients of functions in $L^{2,p}(\mathbb{T})$ for $p>2$. 

\subsection{Morrey-type spaces}
\label{subsec:morrey}
Let $0<q<\infty$ and $\phi$ be a non-negative function on $\mathbb{R}_+$. Define the Morrey-type space $M_{q,\phi}$ by $$\norm{F}_{M_{q,\phi}}:=\sup_{B_R}  \phi(R) \left(\int_{B_R} |F|^q\right)^\frac{1}{q},$$ where the supremum is taken over all balls in $\mathbb{R}^d$.
\begin{proposition}\label{morrey proposition}
    The largest right-admissible space $Y$ for which the inequality
    
    \begin{equation}
\norm{(\widehat{f}\,)^\circ }_{M_{q,\phi}} \leq \norm{f^\circ}_Y
    \end{equation}
    or, equivalently, 
    \begin{equation}
        \norm{\widehat{f}\,}_{M_{q,\phi}} \leq \norm{f}_Y,
    \end{equation} hold for any  $f\in L^1$, satisfies
    
     \begin{enumerate}[label=(\roman*)]
          
     \item if $0<q<2,$  \begin{equation}
\norm{f}_Y \approx_q \sup _{R>0}  \frac{\phi(R)}{R^{\frac{d}{2}}}  \left(\int_0 ^{R^d} \left(\int_0 ^t \left ( \int_0 ^{r^{-1}} f^* \right)^2 dr \right)^{\frac{q}{2}}dt\right)^{\frac{1}{q}};
 %\lesssim \norm{G^*v_*}_p,
 \end{equation}
 
 \item if $2\leq q<\infty,$
 \begin{equation}
\norm{f}_Y \approx_q  \sup _{R>0}  \phi(R) \left(\int_0 ^{R^d}  \left( \int_0 ^{t^{-1}} f^*   \right)^{q} dt\right)^{\frac{1}{q}}.
 \end{equation}
 \end{enumerate}
\end{proposition}
\begin{proof}
   By the translation invariance of $|\widehat{f}\,|$  we deduce that 
    \begin{align*}
        \sup_{\norm{f}_Y=1} \norm{\widehat{f\,}}_{M_{q,\phi}}&= \sup_{\norm{f}_Y=1} \sup_{B_R}  \phi(R) \left(\int_{B_R} |\widehat{f}|^q\right)^\frac{1}{q} =  \sup_{{B_R}} \sup_{\norm{f}_Y=1}  \phi(R) \left(\int_{B_R} |\widehat{f}|^q\right)^\frac{1}{q}\\
        &=\sup_{R>0} \sup_{\norm{f}_Y=1}  \phi(R)\left(\int_{|\xi|\leq R} |\widehat{f}|^q  \right)^\frac{1}{q}.
    \end{align*}
 For any right-admissible $Y$, using   Theorem \ref{theorem:mainabstract} and Corollary  \ref{coro:method} for any fixed $R$, we deduce that 
    \begin{align*}
        \sup_{\norm{f}_Y=1}  \left(\int_{|\xi|\leq R} |\widehat{f}|^q  \right)^\frac{1}{q}
        \approx_q \sup_{\norm{f^\circ}_Y=1} \left(\int_{|\xi|\leq R} |\widehat{f}|^{\circ,q}  \right)^\frac{1}{q}.%\approx_q  \sup_{\norm{G}_Y=1} 
    %\sup_{R>0}\sup_{ F\prec G} \phi(R)\left(\int_0^{R^d} F^{*,q}  \right)^\frac{1}{q}
    % \approx_q
% \sup_{\norm{f^\circ}_Y=1}     \|f\|_{\widetilde{X}_R},
    \end{align*}
Finally, in view of Corollary \ref{coro:optiHL} with $u(x)=\mathbbm{1}_{[0,R]}(|x|)$, 
    \begin{enumerate}[label=(\roman*)]
    \item if $0<q<2,$  \begin{equation}
 \sup_{\norm{f}_Y=1}  \left(\int_{|\xi|\leq R} |\widehat{f}|^q  \right)^\frac{1}{q}\approx_q \sup_{\norm{f}_Y=1} R^{-\frac{d}{2}}  \left(\int_0 ^{R^d} \left(\int_0 ^t \left ( \int_0 ^{r^{-1}} f^* \right)^2 dr \right)^{\frac{q}{2}}dt\right)^{\frac{1}{q}};
 %\lesssim \norm{G^*v_*}_p,
 \end{equation}
 
 \item if $2\leq q<\infty,$
 \begin{equation}
  \sup_{\norm{f}_Y=1}  \left(\int_{|\xi|\leq R} |\widehat{f}|^q  \right)^\frac{1}{q}\approx_q \sup_{\norm{f}_Y=1}   \left(\int_0 ^{R^d}  \left( \int_0 ^{t^{-1}} f^*   \right)^{q} dt\right)^{\frac{1}{q}}.
 \end{equation}
\end{enumerate}
\end{proof}

\begin{remark}
\begin{enumerate}[label=(\roman*)]
  \item If  $\phi(R)= R^{-\lambda}$ with $0\leq \lambda \leq d/q$, then $M_{q,\phi}$ is the classical Morrey space, that is,
$M_{q,\phi}= M_q^\lambda
$, $0< q< \infty$.
By Proposition \ref{morrey proposition},
the following estimates are optimal among  all right-admissible spaces $Y$ for which 
 the inequality
$\norm{\widehat{f}\,}_{M_{q,\phi}} \lesssim \norm{f}_Y$ holds:
for
$0< \lambda < d/q$ and 
$\frac 1s=
\frac 1q- \frac \lambda d$,%we recover the known inequalities  
\begin{equation*}
\norm{\widehat{f}\,}_{M_{q,\phi}} \leq \norm{(\widehat{f}\,)^\circ }_{M_{q,\phi}} \approx_{q,s}  \norm{\widehat{f}\,}_
{s,\infty}
 \lesssim_s \norm{{f}}_
{s',\infty},\qquad s>2;
\end{equation*}
 $$\norm{\widehat{f}\,}_{M_{q,\phi}} 
 \lesssim\norm{\widehat{f}\,}_
{s,\infty}
 \lesssim\norm{{f}}_
{2},\qquad s=2.$$
If we allow for non-admissible spaces $Y$,
then it is possible to derive sharper inequalities 
for the Morrey spaces  by using non-rearrangement techniques as in \cite{Holland}. See also \cite{DEBERNARDIPINOS2024110522} for similar results for a wider class of Morrey type spaces.

\item 
Setting $\phi(R)= {R^{-d}(1+\log_+(1/R)})^{-1}$, we have
$$\norm{F^\circ}_{M_{1,\phi}}=\sup_{B_R} \frac{1}{R^d
(1+\log_+ \frac{1}{R})}
  \int_{B_R} F^\circ=:\norm{F}_{\exp L},$$ which is the exponential Orlicz space. 
Observing that
$$\qquad \quad\sup _{R>0}  \frac{1}{R^{\frac{3d}{2}} (1+\log_+ \frac{1}{R})}  \int_0 ^{R^d} \left(\int_0 ^t \Big ( \int_0 ^{r^{-1}} F^* \Big)^2 dr \right)^{\frac{1}{2}}dt\approx
\sup_{R>0} \frac{1}{
(1+\log_+R)
}   \int_0^R F^*,
$$
  we arrive at the optimal Fourier inequality for the exponential Orlicz space:
\begin{equation}
\label{eq:expL}
\norm{\widehat{F}\,}_{\exp L} \lesssim 
%\norm{{F^*}}_{M_{1,\psi}}==
\sup_{R>0} \frac{1}{
(1+\log_+R)
}   \int_0^R F^*.
%, \quad \psi(R)= {\log^{-1}_+R}.
 \end{equation}
%$\psi(R)= {\log^{-1}_+(1/R)}$, 
%in the one-dimensional case for simplicity, we have
Moreover,
since $$\norm{F}_{(\exp L)'}\approx \norm{F}_{L \log L}:=\int_0^\infty F^*(t)\ (1+\log_+(1/t)) dt,$$ we can deduce a  Hardy-type estimate (cf. \cite[Corollary 7.23, p. 342]{garcia} and 
\cite[Chapter XII, p. 158]{z}): 
\begin{equation}
\norm{\widehat{F}\,}_{\infty}+
\int_1^\infty(\widehat{F})^{*}(t)\frac{dt}{t}
\lesssim 
\norm{{F}}_{L \log L},
 \end{equation}
% where the norm of $Y$ is given by
% $$\norm{G}_{Y} = G^*(0) +\int_1^\infty G^*(t) \frac{dt}{t}, $$
 and the left-hand side norm  is optimal.

\begin{comment}
the smallest left-admissible space $X$ which contains the Fourier transform of the elements of $L\log L$ is the associate space of $Y$, where 
$$\norm{F}_Y \approx \sup _{R>0}  \frac{1}{R^{\frac{3}{2}}\log_+(1/R) }  \int_0 ^{R} \Big(\int_0 ^t \Big ( \int_0 ^{r^{-1}} F^* \Big)^2 dr \Big)^{\frac{1}{2}}dt.$$
\end{comment}

\begin{comment}
To the best of our knowledge,  an explicit expression of  $\norm{G}_{Y'}$ is not known.
We note that for each $G$ its norm in the associate space  $\norm{G}_{Y'}$ is the best constant $C_G:=\norm{G}_{Y'}$ in the inequality 
$$ 
\int_0 ^ \infty F^* G^* \leq C_G \sup _{R>0}  \frac{1}{R^{\frac{3}{2}}\log_+(1/R) } \int_0 ^{R} \Big(\int_0 ^t \Big ( \int_0 ^{r^{-1}} F^* \Big)^2 dr \Big)^{\frac{1}{2}}dt.
$$
\end{comment}

\item As the previous example shows, in light of Corollary \ref{coro:optiHL}, 
it is of interest to find the associates of the norms
$$\left(\int_0 ^\infty u^{*,q}(t) \left( \int_0 ^{t^{-1}} F^*   \right)^{q} dt\right)^{\frac{1}{q}}$$ and
$$  \left(\int_0 ^\infty U^{\frac{q^{\sharp}}{2}}(t) u^{*,q}(t) \xi^{-\frac{q^{\sharp}}{2}}(t) t^{-\frac{q}{2}} \left(\int_0 ^t \left ( \int_0 ^{r^{-1}} F^* \right)^2 dr \right)^{\frac{q}{2}}\right)^{\frac{1}{q}},$$ as these yield
the sharpest rearrangement-invariant inequalities of the type
$\norm{\widehat{f\,}}_X \lesssim \norm{fu^{-1}}_{q'}$ for a given $u$.

By \cite[Remark 6.3]{gogpick}, the associate of the former is given by 
$\norm{wf^{**}}_{q'}
<\infty$
with $w(t)=U^{-1/q}(1/t)t^{-1/q'}$. We have not been able to find an expression for the latter in the literature.
\end{enumerate}
\end{remark}

\begin{comment}

Unfortunately, we have not been able to find an explicit expression for this space, so we ask \newline
\textit{Question 1:} Find an expression for $\norm{G}_{Y'}$. Equivalently, for a given $G$ estimate the best constant $C_G$ in the inequality 
$$ \int_0 ^ \infty F^* G^* \leq C_G \sup _{R>0}  \frac{1}{\log_+(1/R) R^{\frac{3}{2}}}  \int_0 ^{R} \left(\int_0 ^t \left ( \int_0 ^{r^{-1}} F^* \right)^2 dr \right)^{\frac{1}{2}}
% =\norm{f}_{\Gamma_q(v)}:
=\norm{vF^{**}}_q.
$$

Similarly, in light of Corollary \ref{coro:optiHL}, it is interesting to study the associates of the quantities appearing in the corollary. The associate space of 
$$\left(\int_0 ^\infty u^{*,q}(t) \left( \int_0 ^{t^{-1}} F^*   \right)^{q} dt\right)^{\frac{1}{q}}
=\norm{vF^{**}}_q<\infty$$
with 
$v(t)=u^{*,q}(1/t)t^{q-2}$
is given by 
$\norm{wf^{**}}_{q'}
<\infty$
with $w(t)=U^{-1/q}(1/t)t^{-1/q'};$
see \cite[Remark 6.3]{Gogatishvili2006}.

For the quantity in item (i) we have not found any result in the literature. Thus, we pose a second question.

\textit{Question 2:} For a given $G$ and $q\geq 1$ estimate the best constant $C_G$ in the inequality 
$$ \int_0 ^ \infty F^* G^* \leq C_G \left(\int_0 ^\infty U^{\frac{q^{\sharp}}{2}}(t) u^{*,q}(t) \xi^{-\frac{q^{\sharp}}{2}}(t) t^{-\frac{q}{2}} \left(\int_0 ^t \left ( \int_0 ^{r^{-1}} F^* \right)^2 dr \right)^{\frac{q}{2}}\right)^{\frac{1}{q}}.$$

\end{comment}
\subsection{Lorentz spaces}

\label{subsection:FS}

Recall that, for $1\leq p,q\leq \infty$ with $p>1$, the Lorentz function and sequence spaces $L^{p,q}:= L^{p,q}(\mathbb{T})$ and $\ell^{p,q}:= L^{p,q}(\mathbb{N})$, respectively, are given by the norms $$\norm{f}_{p,q}:=\left(\int_0^1 (t^{\frac{1}{p}}f^{**}(t))^q \frac{dt}{t}\right)^\frac{1}{q}\approx_{p,q} \left(\int_0^1 (t^{\frac{1}{p}}f^{*}(t))^q \frac{dt}{t}\right)^\frac{1}{q},$$ 

$$\norm{a}_{p,q}:=\left(\sum_{n=1}^\infty  (n^{\frac{1}{p}}a_n^{**})^q \frac{1}{n} \right)^\frac{1}{q}\approx_{p,q} \left(\sum_{n=1}^\infty  (n^{\frac{1}{p}}a_n^{*})^q \frac{1}{n} \right)^\frac{1}{q}.$$
As we mentioned in the introduction, 
% It is a well-known consequence of the Hardy-Littlewood inequality that, for $1\leq q \leq \infty$ and $1<p<2$,
 if $f\in L^{p,q}$,
 $1<p<2$, $1\leq q \leq \infty$,
 then $(\widehat{f}(n))_{n\in \mathbb{Z}}\in \ell^{p',q}$, and, moreover, $\ell^{p',q}$ is the smallest rearrangement invariant space for which this is true. This follows from the fact that for any even decreasing sequence 
$a:=(a_n)_{n\in \mathbb{Z}},
$ the  function  $f(x)=\sum_{n\in \mathbb{Z}} a_n e^{2 \pi inx}$ satisfies  $\|f\|_{p,q}\approx _{p,q} \|a_n\|_{p',q}$
  provided that $1<p<\infty$,  see \cite{hunt}. 

When $L^{p,q}\subset L^2$,  the smallest r.i. sequence space which contains the Fourier coefficients of every function in $L^{p,q}$ is $\ell^2$ (in fact, more is true, see \cite{KKdL}). 

The corresponding question for $L^{2,q}$ with $q>2$ was raised by Bochkarev \cite{Bochkarev} in 1997,
%It is thus natural to consider the same question for $L^{2,q}$ with $q>2$. 
%The study of this problem was initiated by Bochkarev \cite{Bochkarev},
who proved 
\begin{theorem}
    Let $2<p\leq \infty$, then for any $f\in L^{2,p}(\mathbb{T})$
    \begin{equation}
         \sup_n \frac{1}{ \log^{\frac{1}{p^{\sharp}}}(n+1)}\left(\sum_{j=1}^n \widehat{f}(n)^{*,2 } \right)^{\frac{1}{2}}\lesssim_p \norm{f}_{2,p}.
    \end{equation}
\end{theorem}
It is worth mentioning that 
 already in 1968 Herz \cite[Proposition 3.5]{herz} proved that $\|\widehat{f}(n)\|_{2,q}\lesssim \|{f}\|_{2,p}$ holds if and only if $p\le 2\le q.$
For further results in this direction see \cite{oscar}.

We answer this question and we show that 
 %The following result shows that 
 the sharp estimate for the Fourier coefficients of $L^{2,p}$ functions is given by
\begin{equation}
\label{eq:sharpboch}
\left(\sum_{n=1}^\infty \left(\sum_{j=1}^n \widehat{f}(n)^{*,2 } \right)^{\frac{p}{2}} \frac{1}{n \log^{\frac{p}{2}}(n+1)}\right)^{\frac{1}{p}}\lesssim_p \norm{f}_{{2,p}}, \quad 2<p\le\infty.
\end{equation}
More precisely, we obtain
\begin{theorem}
\label{th:52}
    Assume that $X$ is a rearrangement-invariant function space and let $2<p\leq\infty$. Then,   the following are equivalent:
    \begin{enumerate}[label=(\roman*)]
        \item for any  $f\in L^1$, the inequality \begin{equation}
    \label{ineq:2p1}
        \norm{\widehat{f}(n)}_X \lesssim \norm{|x|^{\frac{1}{p^{\sharp}}}f(x)}_{p}
    \end{equation}
        holds;
        \item for any $f\in L^1$, the inequality \begin{equation}
    \label{ineq:2p2}
         \norm{\widehat{f}(n)}_X \lesssim \norm{f}_{{2,p}}
    \end{equation} holds;
    \item
    for any sequence $b$, the inequality
    \begin{equation}    \label{ineq:2p3}
        \norm{b}_X \lesssim_p \norm{b}_{ {\Theta}_{2,p
        }}:= \left(\sum_{n=1}^\infty \left(\sum_{j=1}^n b_j^{*,2 } \right)^{\frac{p}{2}} \frac{1}{n \log^{\frac{p}{2}}(n+1)}\right)^{\frac{1}{p}}
    \end{equation}  holds. 
    \end{enumerate}
    In consequence, $\Theta_{2,p}$ is the smallest r.i. space which contains the Fourier coefficients of every function in $L^{2,p}$ for $p>2$.
\end{theorem}
For the proof, it is convenient to first obtain the dual result, that is, determining the largest r.i.
sequence space whose Fourier series belong to
$L^{2,q}$ for $q<2$,
a result that is interesting in its own right.

\begin{theorem}
\label{corollary:HL1}
    Let $1\leq q<2$ and assume that $Y$ is a rearrangement-invariant function space. Then,  the following are equivalent:
    \begin{enumerate}[label=(\roman*)]
        \item for any sequence $a\in \ell^1$, the inequality \begin{equation}
    \label{ineq:2q1}
        \norm{ |x|^{-\frac{1}{q^{\sharp}}}\sum_{n\in \mathbb{Z}} a_n e^{2 \pi i n x}}_{q}\lesssim \norm{a}_{Y}
    \end{equation} holds;
        
        \item for any sequence $a\in \ell^1$, the inequality 
        \begin{equation}
    \label{ineq:2q2}
        \norm{\sum_{n\in \mathbb{Z}} a_n e^{2 \pi i n x}}_{{2,q}} \lesssim \norm{a}_{Y}
    \end{equation} holds;
    \item for any sequence $a$, the inequality
    \begin{equation*}
        \left(\sum_{n=1}^\infty \left(\sum_{j=n}^\infty a_j^{**, 2}\right)^{\frac{q}{2}} \frac{1}{n \log^{\frac{q}{2}}(n+1)}\right)^{\frac{1}{q}} =:\norm{a}_{\Gamma_{2,q}}\lesssim_q \norm{a}_Y
    \end{equation*} holds. 
    \end{enumerate}

\end{theorem}
\begin{remark}\label{starvstwostars}
    Note that, for $1\leq q<2$,
   \begin{equation}
    \label{eq:starvstwostars}
    \norm{a}_{\Gamma_{2,q}}\approx_q \left(\sum_{n=1}^\infty \left(\sum_{j=n}^\infty a_j^{*, 2}\right)^{\frac{q}{2}} \frac{1}{n \log^{\frac{q}{2}}(n+1)}\right)^{\frac{1}{q}}.   \end{equation}
    Likewise, the norm $\norm{b}_{ {\Theta}_{2,p
        }}$ can be defined replacing 
        $b_j^{*,2 }$ by $b_j^{**,2}$ in \eqref{ineq:2p3}.
\end{remark}
\begin{proof}[Proof of Theorem \ref{corollary:HL1}]
% By    Theorem
% \ref{theorem:mainabstract} with $\beta=1$
% and $X=L^q(|x|^{-1/q^{\sharp}})$,
% \eqref{ineq:2q1} is equivalent to 
% \eqref{ineq:2q2} and also to
% $$\norm{F}_X\lesssim \norm{A}_Y$$ for any $F\prec A$.
Note that the equivalence $(i) \iff (ii)$ follows from Theorem \ref{theorem:mainabstract} and Corollary \ref{coro:method}.

 The proof of $(i) \iff (iii)$ follows from the discrete version of Corollary \ref{coro:optiHL}, which, after observing that $u(t)=t^{-\frac{1}{q^{\sharp}}} \mathbbm{1}_{[0,1]}$, $U(t)\approx t^{\frac{q}{2}}$ and $\xi(t)\approx t^{\frac{q}{2}} (1+ |\log(t)|^{\frac{q}{q^{\sharp}}})$, yields
    \begin{equation*}
 \left(\sum_{n=1}^\infty \left(\sum_{j=n}^\infty a_j^{**, 2}\right)^{\frac{q}{2}} \frac{1}{n \log^{\frac{q}{2}}(n+1)}\right)^{\frac{1}{q}} \lesssim_q \norm{a}_Y.\end{equation*}
\end{proof}
\begin{proof}[Proof of Theorem \ref{th:52}]
Once again, the equivalence $(i) \iff (ii)$ follows from Theorem \ref{theorem:mainabstract} and Corollary \ref{coro:method}.

\begin{comment}
To begin with, the fact that inequality \eqref{ineq:2p2} holds for $X=\Theta_{2,p}$ is well known (?). Besides, \eqref{ineq:2p2} implies 
\eqref{ineq:2p1} by rearrangement inequalities.
   Now assume that inequality \eqref{ineq:2p1} holds, then
    for any sequence $a$
    \begin{eqnarray*}
        \norm{ |x|^{-\frac{1}{p^{\sharp}}}\sum_{n\in \mathbb{Z}} a_n e^{2 \pi i n x}}_{p'}&=&\sup_{\norm{|x|^{\frac{1}{p^{\sharp}}}f(x)}_p=1} \int_0 ^1 \overline{f(x)} \sum_{n \in \mathbb{Z}} a_n e^{2 \pi i nx} dx \\
                &=&\sum_{n \in \mathbb{Z}} \widehat{f}(n) a_n \leq \norm{\widehat{f}(n)}_X \norm{a}_{X'} \lesssim \norm{a}_{X'},
    \end{eqnarray*} where $X'$ is the associate space of $X$, which has norm $\norm{a}_{X'}:= \sup_{\norm{b}_X=1}\sum_{n \in \mathbb{Z}} b_n a_n.$

    Since $X$ is rearrangement-invariant, so is $X'$. Hence, by Corollary \ref{corollary:HL1},
    $\norm{a}_{\Gamma_{2,p'}} \lesssim \norm{a}_{X'}$ and, equivalently, $\norm{b}_{\Gamma'_{2,p'}} \gtrsim \norm{b}_{X}$. Therefore, it suffices to show that
    $$\sum_{n=1}^\infty a_n^* b_n^* \lesssim \norm{a}_{\Gamma_{2,p'}} \norm{b}_{\Theta_{2,p}}.$$
\end{comment}

We now show that $(ii)\iff (iii)$.
First, we prove the inequality
\begin{equation}
\label{eq:step20}
    \norm{\widehat{f}(n)}_{\Theta_{2,p}} \lesssim_p  \norm{f}_{{2,p}}.\end{equation}   By Corollary \ref{coro:method}, it is enough to show that for $f$ supported on $[0,1]$ one has
\begin{equation*}
\left(\int_1^\infty \left(\int_0 ^x \left(\int_0 ^{t^{-1}} f^*\right)^{2}\right)^{\frac{p}{2}}\frac{dx}{x \log^{\frac{p}{2}}(x+1)}\right)^{\frac{1}{p}} \lesssim_p  \left(\int_0 ^1 (f^{**}(t) t^{\frac{1}{2}})^p \frac{dt}{t}\right)^{\frac{1}{p}},
\end{equation*}  or, equivalently, that both estimates 
\begin{equation*}
f^{**}(1)\left(\int_{1}^\infty \frac{dt}{t \log^{\frac{p}{2}}(t+1)}\right)^{\frac{1}{p}} \lesssim_p \left(\int_0 ^1 (f^{**}(t) t^{\frac{1}{2}})^p \frac{dt}{t}\right)^{\frac{1}{p}}
\end{equation*} and \begin{equation*}
\left(\int_1^\infty \left(\int_{t^{-1}} ^1  f^{**,2}\right)^{\frac{p}{2}} \frac{dt}{t \log^{\frac{p}{2}}(t+1)}\right)^{\frac{1}{p}} \lesssim_p  \left(\int_0 ^1 (f^{**}(t) t^{\frac{1}{2}})^p \frac{dt}{t}\right)^{\frac{1}{p}}
\end{equation*} hold. The former holds because $f^{**}(1)\approx f^{**}(1/2)$. To see that the latter holds, observe that by Hardy's inequality (Lemma \ref{theorem:cont hardy}), for any positive $g$,
\begin{equation*}
\left(\int_1^\infty \left(\int_{t^{-1}} ^1  g\right)^{\frac{p}{2}} \frac{dt}{t \log^{\frac{p}{2}}(t+1)}\right)^{\frac{2}{p}} \lesssim_p  \left(\int_0 ^1 (g(t) t)^{\frac{p}{2}} \frac{dt}{t}\right)^{\frac{2}{p}}.
\end{equation*}
In conclusion, inequality \eqref{eq:step20} holds.

Second, by Lemma \ref{lemma:duality}, in light of the relation $(L^{2,p})'= L^{2,p'}$, any $X$ for which the inequality
$$\norm{\widehat{f}(n)}_X\lesssim \norm{f}_{{2,p}}$$ holds must satisfy
$$ \norm{\sum_{n \in \mathbb{Z}} a_n e^{2 \pi i n x}}_{{2,p'}}\lesssim \norm{a}_{X'},$$ which, by Theorem \ref{corollary:HL1}, implies that $\norm{\cdot}_{\Gamma_{2,p'}}\lesssim_p\norm{\cdot}_{X'}$. 
Hence, the result will follow if we prove that $\norm{\cdot}_{\Gamma_{2,p'}}\approx_p\norm{\cdot}_{(\Theta_{2,p})'}.$ Observe that  inequality \eqref{eq:step20}
 yields $\norm{\cdot}_{\Gamma_{2,p'}}\lesssim_p\norm{\cdot}_{(\Theta_{2,p})'}$, so 
it suffices to prove the reverse inequality.

   Third, to see that $\norm{\cdot}_{\Gamma_{2,p'}}\gtrsim_p\norm{\cdot}_{(\Theta_{2,p})'}$ we use the following equivalent descriptions of the norms of $\Gamma_{2,p'}$ and $\Theta_{2,p}$, which follow from Hardy's inequality: let $(y_k)_{k \geq 0}$ be an increasing sequence such that $y_0=1$ and $\log(y_k) = 2^k$ for $k\geq 1$. Then it is easy to see that

    $$\norm{a}_{\Gamma_{2,p'}} \approx_p  \left(\sum_{k=0} ^\infty 2^{k(1-\frac{p'}{2})} \left(\sum_{j=y_k}^{y_{k+1}} a_j^{**,2}\right)^{\frac{p'}{2}}\right)^{\frac{1}{p'}}$$ and
    $$\norm{b}_{\Theta_{2,p}} \approx_p  \left(\sum_{k=0} ^\infty 2^{k(1-\frac{p}{2})} \left(\sum_{j=y_k}^{y_{k+1}} b_j^{*, 2}\right)^{\frac{p}{2}}\right)^{\frac{1}{p}}.$$
    Hence,
    $$\sum_{n=1}^\infty a_n^* b_n^* \ \leq  \sum_{k=0}^\infty \left(\sum_{j=y_k}^{y_{k+1}} a_j^{*, 2}\right)^{\frac{1}{2}} \left(\sum_{j=y_k}^{y_{k+1}} b_j^{*, 2}\right)^{\frac{1}{2}} \lesssim_p  \norm{a}_{\Gamma_{2,p'}} \norm{b}_{\Theta_{2,p}}.  $$

    This shows that $\Theta_{2,p}= \left(\Gamma_{2,p'}\right)'$. The proof is now complete.
\end{proof}

\begin{proof}[Proof of Remark \ref{starvstwostars}]
Equivalence \eqref{eq:starvstwostars}
follows from the proof of Theorem \ref{th:52}.
The second part of the remark follows by Hardy's inequality, using Lemma
\ref{theorem:cont hardy2} with $\mathfrak{p}=\mathfrak{q}=2$.
\end{proof}

\subsection{Weighted Lorentz spaces}
Let $0<p\leq \infty $ and $u,v$ be  non-negative weight  functions. Define the weighted Lorentz spaces by
\begin{equation*}
    \norm{f}_{\Lambda_p(u)}:=\norm{uf^*}_p
\end{equation*}
and 
\begin{equation*}
    \norm{f}_{\Gamma_p(v)}:=\norm{vf^{**}}_p.
\end{equation*} In \cite[Corollary 5.3]{sinnamon2003} it is proved that, for $p\leq 2$, the Fourier transform is bounded from $\Gamma_p(v)$ to $\Lambda_2(u)$ if and only if any operator of joint strong type $(1,\infty;2,2)$ is also bounded from $\Gamma_p(v)$ to $\Lambda_2(u)$. 

Theorem \ref{theorem:mainabstract} allows us 
to characterize  the   boundedness of the Fourier transform in the reverse direction for any $p$ and $q.$ 
Indeed, since $\Gamma_q(u)$ and $\Lambda_p(v)$ are left- and right-admissible, respectively,  we deduce

%We note, however, that Theorem \ref{theorem:mainabstract} does not apply here, because  $\norm{f}_{\Lambda_2(u)}$ need not be left-admissible for a $u$ which is not non-increasing. On the other hand, since $\Gamma_q(u)$ and $\Lambda_p(v)$ are left- and right-admissible, respectively, from Theorem \ref{theorem:mainabstract} we deduce 
\begin{corollary}
    Let $0< p,q\leq \infty$. Then, the Fourier transform is bounded from $\Lambda_p(v)$ to $\Gamma_q(u)$ if and only if any operator of joint strong type $(1,\infty;2,2)$ is also bounded.
\end{corollary}
Here we note,
by Theorem \ref{theorem:mainabstract}
and Corollary \ref{coro:method}, 
that the 
characterization of the optimal right-admissible space $Y$ such that 
$$\norm{\widehat{f}\,}_{\Gamma_q(u)}\lesssim \norm {f}_{Y}$$ 
is given by
$$\norm{f}_Y:=K(f),$$ where $K(f)$ is the smallest constant for which the inequality
$$
\left(\int_0^\infty u^{*,q}(x) \left(\frac{1}{x}\int_0^x F^{*} \right)^q dx \right)^\frac{1}{q}
%=\norm{F^{*}}_{\Gamma_q(u)}
\leq K(f) \sup _{x>0} \frac{\left(\int_0 ^x F^{*,2}\right)^\frac{1}{2}}{\left(\int_0 ^x \left(\int_0 ^{t^{-1}} f^*\right)^2 dt\right)^\frac{1}{2}}$$ holds for any $F$.

\begin{comment}

Here we recall  the characterization of 
the optimal right-admissible space $Y$ such that 
$$\norm{\widehat{f}\,}_{\Gamma_q(u)}\lesssim \norm {f}_{Y}$$ 
%is valid follows from  Theorem \ref{theorem:mainabstract}.
given by 
Corollary     \ref{coro:mainbis}:
$\norm{f}_{Y}= \sup \norm{T(f)}_{\Gamma_q(u)}$,
where the supremum is taken over all linear operators $T$ of joint strong type $(1,\infty;2,2)$.
Equivalently, 
$$\norm{f}_Y:=K(f),$$ where $K(f)$ is the smallest constant for which the inequality
$$
\left(\int_0^\infty u^{*,q}(x) \left(\frac{1}{x}\int_0^x F^{*} \right)^q dx \right)^\frac{1}{q}
%=\norm{F^{*}}_{\Gamma_q(u)}
\leq K(f) \sup _{x>0} \frac{\left(\int_0 ^x F^{*,2}\right)^\frac{1}{2}}{\left(\int_0 ^x \left(\int_0 ^{t^{-1}} f^*\right)^2 dt\right)^\frac{1}{2}}$$ holds for any $F$.

\end{comment}

Finally, the optimal space $Y$ for which the Fourier inequality 
$$\norm{\widehat{f}\,}_{\Lambda_q(u)}\lesssim \norm {f}_{Y}$$ holds for $0<q< \infty$ and non-increasing $u$
is described in 
Corollary \ref{coro:optiHL}, cf.
\cite[Lemma 2.3]{GRAFAKOS2022109295}.

\bibliographystyle{amsplain}
\bibliography{main}

\providecommand{\bysame}{\leavevmode\hbox to3em{\hrulefill}\thinspace}
\providecommand{\MR}{\relax\ifhmode\unskip\space\fi MR }
% \MRhref is called by the amsart/book/proc definition of \MR.
\providecommand{\MRhref}[2]{%
  \href{http://www.ams.org/mathscinet-getitem?mr=#1}{#2}
}
\providecommand{\href}[2]{#2}
\begin{thebibliography}{10}

\bibitem{Alg}
N.~Aguilera and E.~Harboure, \emph{On the search for weighted norm inequalities for the {F}ourier transform}, Pacific J. Math. \textbf{104} (1983), no.~1, 1--14.

\bibitem{albiac2006topics}
F.~Albiac and N.~Kalton, \emph{Topics in {B}anach space theory}, Graduate Texts in Mathematics, vol. 233, Springer, 2006.

\bibitem{beckner2008}
W.~Beckner, \emph{Pitt's inequality with sharp convolution estimates}, Proc. Amer. Math. Soc. \textbf{136} (2008), no.~5, 1871--1885.

\bibitem{BHJFAA}
J.~Benedetto and H.~Heinig, \emph{Weighted {F}ourier inequalities: new proofs and generalizations}, J. {F}ourier Anal. Appl. \textbf{9} (2003), 1--37.

\bibitem{nachr}
J.~Benedetto, H.~Heinig, and R.~Johnson, \emph{Weighted {H}ardy spaces and the {L}aplace transform. {II}}, Math. Nachr. \textbf{132} (1987), 29--55.

\bibitem{bennett1988interpolation}
C.~Bennett and R.~Sharpley, \emph{Interpolation of operators}, Pure and Applied Mathematics, Academic Press, 1988.

\bibitem{lacey1980notes}
G.~Bennett, \emph{Lectures on matrix transformations of $\ell^p$ spaces}, Notes in {B}anach Spaces, Semin. and Courses, Austin/Tex. 1975-79, 1980, pp.~39--80.

\bibitem{Bochkarev}
S.~Bochkarev, \emph{Hausdorff-{Young}-{Riesz} theorem in {Lorentz} spaces and multiplicative inequalities}, Proc. Steklov Inst. Math. \textbf{219} (1997), 96--107.

\bibitem{bui}
H.-Q. Bui, \emph{Bernstein's theorem on weighted {Besov} spaces}, Forum Math. \textbf{9} (1997), no.~6, 739--750.

\bibitem{Calderon}
A.~Calder{\'o}n, \emph{Spaces between {{\(L^ 1\)}} and {{\(L^ \infty\)}} and the theorem of {Marcinkiewicz}}, Stud. Math. \textbf{26} (1966), 273--299.

\bibitem{decarli}
L.~De~Carli, D.~Gorbachev, and S.~Tikhonov, \emph{Pitt inequalities and restriction theorems for the {F}ourier transform}, Rev. Mat. Iberoam. \textbf{33} (2017), no.~3, 789--808.

\bibitem{gradient}
\bysame, \emph{Weighted gradient inequalities and unique continuation problems}, Calc. Var. Partial Differ. Equ. \textbf{59} (2020), no.~3, 24.

\bibitem{KKdL}
K.~de~Leeuw, Y.~Katznelson, and J.-P. Kahane, \emph{Sur les coefficients de {Fourier} des fonctions continues}, C. R. Acad. Sci., Paris, S{\'e}r. A \textbf{285} (1977), 1001--1003.

\bibitem{DEBERNARDIPINOS2024110522}
A.~{Debernardi Pinos}, E.~Nursultanov, and S.~Tikhonov, \emph{{F}ourier inequalities in {M}orrey and {C}ampanato spaces}, J. Funct. Anal. \textbf{287} (2024), no.~7, 110522.

\bibitem{oscar}
O.~Domínguez and M.~Veraar, \emph{Extensions of the vector-valued {Hausdorff}-{Young} inequalities}, Math. Z. \textbf{299} (2021), no.~1--2, 373--425.

\bibitem{garcia}
J.~Garc{\'{\i}}a-Cuerva and J.L. Rubio~de Francia, \emph{Weighted norm inequalities and related topics}, North-Holland Math. Stud., Elsevier, Amsterdam, 1985.

\bibitem{gogpick}
A.~Gogatishvili and L.~Pick, \emph{Discretization and anti-discretization of rearrangement-invariant norms}, Publ. Mat. \textbf{47} (2003), no.~2, 311--358.

\bibitem{Gogatishvili2006}
\bysame, \emph{Embeddings and duality theorems for weak classical {L}orentz spaces}, Canad. Math. Bull. \textbf{49} (2006), no.~1, 82--95.

\bibitem{stepanov}
A.~Gogatishvili and V.~Stepanov, \emph{Reduction theorems for weighted integral inequalities on the cone of monotone functions}, Russian Math. Surveys \textbf{68} (2013), no.~4, 597--664.

\bibitem{indiana}
D.~Gorbachev, E.~Liflyand, and S.~Tikhonov, \emph{Weighted norm inequalities for integral transforms}, Indiana Univ. Math. J. \textbf{67} (2018), no.~5, 1949--2003.

\bibitem{GRAFAKOS2022109295}
L.~Grafakos, M.~Mastyło, and L.~Slavíková, \emph{A sharp variant of the {Marcinkiewicz theorem with multipliers in Sobolev spaces of Lorentz type}}, J. Funct. Anal. \textbf{282} (2022), no.~3, 109295.

\bibitem{hardylittlewood}
G.~Hardy and J.~Littlewood, \emph{Some new properties of {Fourier} constants}, Math. Ann. \textbf{97} (1926), 159--209.

\bibitem{hausdorff}
F.~Hausdorff, \emph{Eine {Ausdehnung} des {Parsevalschen} {Satzes} {\"u}ber {Fourierreihen}}, Math. Z. \textbf{16} (1923), 163--169.

\bibitem{h}
H.~Heinig, \emph{Weighted norm inequalities for classes of operators}, Indiana Univ. Math. J. \textbf{33} (1984), 573--582.

\bibitem{herz}
C.~Herz, \emph{Lipschitz spaces and {Bernstein}'s theorem on absolutely convergent {Fourier} transforms}, J. Math. Mech. \textbf{18} (1968), 283--323.

\bibitem{Holland}
F.~Holland, \emph{Harmonic analysis on amalgams of ${L}^p$ and $\ell^q$}, J. Lond. Math. Soc. \textbf{10} (1975), no.~3, 295--305.

\bibitem{hunt}
R.~Hunt, \emph{On {L}(p,q) spaces}, Enseign. Math. (2) \textbf{12} (1966), 249--276.

\bibitem{jodeit}
M.~Jodeit and A.~Torchinsky, \emph{Inequalities for {F}ourier transforms}, Studia Math. \textbf{37} (1971), no.~3, 245--276.

\bibitem{jurkat}
W.~Jurkat and G.~Sampson, \emph{On maximal rearrangement inequalities for the {Fourier} transform}, Trans. Am. Math. Soc. \textbf{282} (1984), 625--643.

\bibitem{JS1}
\bysame, \emph{On rearrangement and weight inequalities for the {F}ourier transform}, Indiana Univ. Math. J. \textbf{33} (1984), 257--270.

\bibitem{kerman}
R.~Kerman, R.~Rawat, and R.K. Singh, \emph{The {F}ourier transform on rearrangement-invariant spaces}, J. {F}ourier Anal. Appl. \textbf{30} (2024), no.~42.

\bibitem{kufner2007hardy}
A.~Kufner, L.~Maligranda, and L.E. Persson, \emph{The {H}ardy inequality: About its history and some related results}, Vydavatelsk{\`y} servis, 2007.

\bibitem{Krepela2022}
M.~Křepela, Z.~Mihula, and H.~Turčinová, \emph{Discretization and antidiscretization of {L}orentz norms with no restrictions on weights}, Rev. Mat. Complut. \textbf{35} (2022), 615--648.

\bibitem{Lindenstrauss1996}
J.~Lindenstrauss and L.~Tzafriri, \emph{Classical {B}anach spaces {II}: Function spaces}, Springer, 1976.

\bibitem{Maurey1973}
B.~Maurey, \emph{Théorèmes de factorisation pour les opérateurs linéaires à valeurs dans un espace \({L}^p({\Omega}, \mu)\), $0 < p \leq \infty$}, Séminaire d’analyse fonctionnelle (Polytechnique) (1973), no.~15, 1--8.

\bibitem{Muc1}
B.~Muckenhoupt, \emph{Weighted norm inequalities for the {F}ourier transform}, Trans. Amer. Math. Soc. \textbf{276} (1983), 729--742.

\bibitem{nielsen}
M.~Nielsen, \emph{An example of an almost greedy uniformly bounded orthonormal basis for {{\(L_p(0,1)\)}}}, J. Approx. Theory \textbf{149} (2007), no.~2, 188--192.

\bibitem{nielsen2}
\bysame, \emph{Trigonometric quasi-greedy bases for {{\(L^p(\mathbf T;w)\)}}}, Rocky Mt. J. Math. \textbf{39} (2009), no.~4, 1267--1278.

\bibitem{olevskiibook}
A.~Olevskii, \emph{Fourier series with respect to general orthogonal systems}, Ergeb. Math. Grenzgeb., vol.~86, Springer-Verlag, Berlin, 1975.

\bibitem{popova}
L.~Persson, O.~Popova, and V.~Stepanov, \emph{Two-sided {H}ardy-type inequalities for monotone functions}, Dokl. Math. \textbf{80} (2009), no.~3, 814--817.

\bibitem{pitt}
H.~Pitt, \emph{Theorems on {F}ourier series and power series}, Duke Math. J. \textbf{3} (1937), no.~4, 747--755.

\bibitem{rastegari}
J.~Rastegari and G.~Sinnamon, \emph{Fourier series in weighted {Lorentz} spaces}, J. Fourier Anal. Appl. \textbf{22} (2016), no.~5, 1192--1223.

\bibitem{sin}
\bysame, \emph{Weighted {F}ourier inequalities via rearrangements}, J. {F}ourier Anal. Appl. \textbf{24} (2018), 1225--1248.

\bibitem{riesz}
M.~Riesz, \emph{Sur les maxima des formes bilin{\'e}aires et sur les fonctionelles lin{\'e}aires}, Acta Math. \textbf{49} (1927), 465--497.

\bibitem{sinnamon2003}
G.~Sinnamon, \emph{The {F}ourier transform in weighted {L}orentz spaces}, Publ. Mat. \textbf{47} (2003), no.~1, 3--29.

\bibitem{stein}
E.~Stein, \emph{Interpolation of linear operators}, Trans. Amer. Math. Soc. \textbf{83} (1956), no.~2, 482--492.

\bibitem{young}
W.~Young, \emph{On the determination of the summability of a function by means of its \emph{Fourier} constants}, Proc. Lond. Math. Soc. (2) \textbf{12} (1913), 71--88.

\bibitem{z}
A.~Zygmund, \emph{Trigonometric series. {Volumes} {I} and {II}}, Cambridge University Press, 2002.

\end{thebibliography}

\end{document}